\documentclass[10pt,a4paper]{amsart}
\usepackage[utf8]{inputenc}
\usepackage[T1]{fontenc}
\usepackage{amsmath}
\usepackage{amsfonts}
\usepackage{amssymb}
\usepackage{amsaddr}
\usepackage{enumitem}

\usepackage{amsthm}
\usepackage{color}
\usepackage{bbm}
\usepackage{fullpage}
\usepackage{comment}
\author{Malo Jézéquel}
\date{\today}
\address{DMA, ENS, 45 rue d'Ulm, 75005, Paris \\ (email: malo.jezequel@ens.fr)}
\title{Parameter regularity of dynamical determinants of expanding maps of the circle and an application to linear response}
\newcommand{\B}{\mathcal{B}}
\renewcommand{\(}{\left(}
\renewcommand{\)}{\right)}

\newcommand{\F}{\mathbb{F}}

\newcommand{\N}{\mathbb{N}}
\newcommand{\R}{\mathbb{R}}

\newcommand{\C}{\mathbb{C}}
\newcommand{\M}{\mathcal{M}}
\newcommand{\op}{\text{Op}}
\newcommand{\A}{\mathcal{A}}
\newcommand{\D}[2]{\mathbb{D}\left(#1,#2\right)}
\renewcommand{\L}{\mathcal{L}}

\newcommand{\s}[3]{\int_{#1} #2 \mathrm{d}#3}
\newcommand{\set}[1]{\left\{#1\right\}}

\newcommand{\drond}[2]{\frac{\partial #1}{\partial #2}}
\newcommand{\ddeux}[3]{\frac{\partial^2 #1}{\partial #2 \partial #3}}

\renewcommand{\phi}{\varphi}

\renewcommand{\bar}[1]{\overline{#1}}
\newcommand{\hra}{\hookrightarrow}
\newcommand{\nhra}{\not{\hookrightarrow}}
\newtheorem{lm}{Lemma}[section]
\newtheorem{prop}[lm]{Proposition}

\newtheorem{cor}[lm]{Corollary}
\newtheorem{thm}[lm]{Theorem}
\theoremstyle{definition}
\newtheorem{rmq}[lm]{Remark}

\begin{document}

\begin{abstract}
In order to adapt to the differentiable setting a formula for linear response proved by Pollicott and Vytnova in the analytic setting, we show a result of regularity of dynamical determinants of expanding maps of the circle. The main tool is the decomposition of a transfer operator as a sum of a nuclear part and a "small" bounded part.
\end{abstract}

\maketitle

\section*{Introduction}

In this work we adapt to the differentiable setting the formula for linear response obtained by Pollicott and Vytnova \cite{Poll} in the analytic setting (based on an idea of Cvitanovic \cite{Cvi}). We recall the main argument: if $\tau  \mapsto T_\tau$ is an analytic curve of analytic expanding maps of the circle, defined on a neighbourhood $\left]-\epsilon,\epsilon\right[$ of $0$, and $g : S^1 \mapsto \R$ is an analytic function then for all $\tau \in \left]-\epsilon,\epsilon\right[$ and $u \in \R$ the map
\begin{equation}\label{determinant}
z \mapsto \exp\(- \sum_{n \geqslant 1} \frac{1}{n} \(\sum_{T_\tau^n x = x} \frac{\exp\(-u \sum_{k=0}^{n-1} g\(T_\tau^k x\)\)}{\left| \(T_\tau^n\)'\(x\)-1 \right|}\)z^n\)
\end{equation}
extends to an entire function $z \mapsto d\(z,u,\tau\)$ and the map $\(z,u,\tau\) \mapsto d\(z,u,\tau\)$ defined this way is analytic. Exploring the properties of this dynamical determinant $d$, Pollicott and Vytnova showed that for all $\tau \in \left]-\epsilon,\epsilon \right[$ we have
\begin{equation}\label{valeur}
\s{S^1}{g}{\mu_\tau} = -\frac{\drond{d}{u}\(1,0,\tau\)}{\drond{d}{z} \(1,0,\tau\)}
,\end{equation}
where $\mu_\tau$ denotes the unique measure invariant for $T_\tau$ which is absolutely continuous with respect to Lebesgue, and thus
\begin{equation}\label{derivee}
\left. \drond{}{\tau}\( \s{S^1}{g}{\mu_\tau} \) \right|_{\tau=0} = -\frac{\ddeux{d}{u}{\tau}\(1,0,0\)}{\drond{d}{z}\(1,0,0\)} + \frac{\ddeux{d}{\tau}{z}\(1,0,0\) \drond{d}{u}\(1,0,0\)}{\(\drond{d}{z}\(1,0,0\)\)^2}
.\end{equation}
Investigating this formula, it is easy to write it in terms of the value of the derivative at $\tau=0$ of $\tau \mapsto T_\tau$ on the periodic points of $T_0$ (see Remark \ref{fin}, in particular formula \eqref{explicite}).

Reading the proof of Lemma 3.1 of \cite{Poll}, it appears that if there is some $R >1$ such that, for all $u,\tau$ sufficiently close to $0$, the map defined by \eqref{determinant} extends on the disc of center $0$ and radius $R$ to a holomorphic function $z \mapsto d\(z,u,\tau\)$), then, if $d$ is $\mathcal{C}^1$, the formula \eqref{valeur} holds and, if $d$ is $\mathcal{C}^2$, the formula \eqref{derivee} holds.

Consequently, our main result is a result of regularity for dynamical determinants of expanding maps of the circle: Theorem \ref{main} which implies in particular (taking $t = \(\tau,u\)$, $ T_t = T_{\tau,u} = T_{\tau}$ and $g_{t,u} = g_{\tau,u} = -u g -\log \left|T_\tau'\right|$) that \eqref{valeur}  and \eqref{derivee} hold under some assumptions of regularity of $\tau \mapsto T_\tau$ and $g$ (see Corollary \ref{adaptation}). The hypotheses of Theorem \ref{main} are not the weakest possible, they have been chosen to make the exposition as simple and self-contained as possible (see remark \ref{weaken}).

To define the dynamical determinant and prove its main properties in the analytic case, one may use the work of Ruelle (see \cite{Ruelle1}). We will use here the approach exposed in the first part of \cite{Bal2} (see also \cite{Tsu}). In particular, the main ingredient of our proof will be the decomposition of a transfer operator (Proposition \ref{fonda}), which is an adaptation of Proposition 3.15 of \cite{Bal2}.

In \S \ref{1}, we set the notations involved in the Paley-Littlewood decomposition.

In \S \ref{2}, we use the Paley-Littlewood decomposition to state Lemma \ref{decloc}, which is a local version of the decomposition mentioned above.

Section \ref{3} is dedicated to the definition of the transfer operator and its decomposition as described above.

In \S \ref{4}, we define a "flat trace" for some operators. It may be proved that it coincides with the flat trace defined in Section 3.2.2 of \cite{Bal2} in most cases.

In \S \ref{5}, we prove some results of regularities of the eigenelements of the transfer operator, using a method due to Gouëzel, Keller and Liverani (see \cite{Ke-Li} and \cite{GLK}).

In \S \ref{6}, we state and prove our main result, Theorem \ref{main} and show how to deduce formulae \eqref{valeur} and \eqref{derivee} from it.

In the analytic setting, one can prove a similar statement in the case of Anosov diffeomorphisms of the torus using the work of Rugh (see \cite{R1} and \cite{R2}). It is likely that the method presented here adapt to the case of differentiable expanding maps in higher dimensions (see Remark \ref{weaken}) and to the case of Anosov diffeomorphisms (using the approach exposed in the second part of \cite{Bal2}).

This work is part of a master degree internship under the supervision of Viviane Baladi.

\section{The Paley-Littlewood decomposition}\label{1}
We shall denote by $\mathcal{S}$ the Schwartz class on $\R$ and by $\mathcal{S}'$ the space of tempered distributions on $\R$. If $s \in \R$ write $ H^{s} = \set{ \phi \in \mathcal{S}' : \left\|\phi\right\|_{H^s} < + \infty} $ where $ \left\| \phi \right\|_{H^s}^2 = \s{\R}{\(1+\xi^2\)^s \left|\widehat{\phi}\(\xi\)\right|^2}{\xi}$, if the Fourier transform $\widehat{\phi}$ of $\phi$ is locally $L^2$ (otherwise set $\left\| \phi \right\|_{H^s}^2 = + \infty$). Recall that if $s,t \in \R$, $\theta \in \left]0,1\right[$ and $\phi \in \mathcal{S}'$, Hölder's inequality implies that
\begin{equation*}
\left\| \phi \right\|_{H^{\theta s + \(1-\theta\) t}} \leqslant \left\| \phi \right\|_{H^s}^{\theta} \left\| \phi \right\|_{H^t}^{1-\theta}
.\end{equation*}
In particular, if $\B$ is some Banach space and $A : \mathcal{B} \to \mathcal{S}'$ is a linear operator then we have
\begin{equation}\label{interpolation}
\left\|A\right\|_{\L\(\B,H^{\theta s + \(1-\theta\) t}\)} \leqslant \left\|A\right\|_{\L\(\B,H^s\)}^\theta \left\|A\right\|_{\L\(\B,H^t\)}^{1-\theta}
,\end{equation}
where the space $\L\(\B,\B'\)$ of bounded linear operator from a Banach $\B$ to another $\B'$ is equipped with the operator norm $\|.\|_{\L\(\B,\B'\)}$.

If $\psi : \R \to \R$ is some $\mathcal{C}^{\infty}$ function of at most polynomial growth, denote by $\op \( \psi\)$ the pseudo-differential operator defined by
\begin{equation*}
\op \(\psi \) \phi \(x\) = \frac{1}{2 \pi} \s{\R}{e^{-i x \xi} \psi\(\xi\) \widehat{\phi}\(\xi\)}{\xi}
\end{equation*}
for $x \in \R$ and $\phi \in \mathcal{S}$. Thus $\op \(\psi\)$ is the operator of multiplication by $\psi$ in Fourier transform. We shall extend $\op \(\psi\)$ to an operator between Sobolev spaces as often as we may.

We can now present the Paley-Littlewood decomposition. Fix a $\mathcal{C}^{\infty}$ even function $\chi : \R \to \left[0,1\right]$ such that
\begin{equation*}
\chi\(x\) = 1 \textrm{ for } x \leqslant 1 \textrm{ and } \chi\(x\) = 0 \textrm{ for } x \geqslant 2
.\end{equation*}
Then for $n \in \N$ define $\psi_n$ by $\psi_0 = \chi$ and
\begin{equation*}
\psi_n\(x\) = \chi\(2^{-n}x\) - \chi\(2^{-n+1}x\)  
\end{equation*}
if $n \geqslant 1$ and $x \in \R$. We shall also need the functions $\tilde{\psi}_n$ defined by $\tilde{\psi}_0\(x\) = \chi\(\frac{x}{2}\)$ and
\begin{equation*}
\tilde{\psi}_n\(x\) = \chi\(2^{-n-1}x\) - \chi\(2^{-n+2}x\)
\end{equation*}
if $n \geqslant 1$ and $x \in \R$. Then the following assertions are easily proved :
\begin{equation*}
\sum_{n \geqslant 1} \psi_n = 1,
\quad
\forall n \geqslant 1, x \in \R : \psi_n\(x\) = \psi_1\(2^{-n+1}x\), \tilde{\psi}_n\(x\) = \tilde{\psi}_1\(2^{-n+1}x\),
\end{equation*}
\begin{equation*}
\forall n \geqslant 1 : \textrm{ supp } \psi_n \subseteq \set{ \xi \in \R : 2^{n-1} \leqslant \left| \xi \right| \leqslant 2^{n+1} } ,
\quad
\forall n \in \N : \psi_n \tilde{\psi}_n = \psi_n ,
\end{equation*}
\begin{equation*}
\forall n,l \in \N : \left|l-n\right| > 1 \Rightarrow \psi_n \psi_l = 0
.\end{equation*}
For all $m \in \N^*$ there is a constant $C_m$ such that for all $n \in \N$
\begin{equation}\label{decroit}
\left\|\psi_n^{\(m\)}\right\|_{L^{\infty}} \leqslant C_m 2^{-n m}
.\end{equation}
Finally, using Plancherel's formula, for all $s \in \R$ there is a constant $C>0$ such that for all $\phi \in H^s$ we have
\begin{equation}\label{equivalent}
C^{-1} \left\| \phi \right\|_{H^s}^2 \leqslant \sum_{n \geqslant 0} 2^{-ns} \left\|\op \(\psi_n\) \phi \right\|_{L^2}^2 \leqslant C \left\|\phi\right\|_{H^s}^2 
.\end{equation}
In particular for all $s,s' \in \R$ there is a constant $C$ such that for all $n \in \N$ we have
\begin{equation}\label{norme}
\left\|\op \(\psi_n\)\right\|_{\L\(H^s,H^{s'}\)} \leqslant C 2^{n\(s'-s\)}
.\end{equation}

\section{Local decomposition of the transfer operator}\label{2}

This section aims at proving the following lemma which is a variation on Lemma 2.22 of \cite{Bal2}. In this statement, $F_t$ must be thought as an inverse branch of an expanding map of the circle $T_t$.

\begin{lm}\label{decloc}
Let $r \in \N^*$ and $N \in \N$. Let $U$ be an open set of $\R^D$ for some integer $D$. Let $t\mapsto f_t \in \mathcal{C}^r\(\R\)$ be a $\mathcal{C}^N$ function on $U$ whose values are supported in a bounded open interval $\left]a,b\right[$. Let $t \mapsto F_t \in \mathcal{C}^{r+1} \(\R\)$ be a $\mathcal{C}^N$ function on $U$ whose values are diffeomorphisms with derivative bounded by $\Lambda$ on $\textrm{ supp } f_t$. For all $\phi \in \mathcal{S}$ and $t \in U$ set
\[
\M_t \phi = f_t.\(\phi \circ F_t\) 
.\]
Then for all $t \in U$ the operator $\M_t$ can be written as a sum
\[
\M_t = \(\M_t\)_b + \(M_t\)_c
\]
such that the following properties hold :
\begin{enumerate}[label=\roman*.]
\item for all $s < r$ there is a constant $c>0$ that only depends of $s$ and $r$ such that for all $t \in U$ :
\[
\left\|\(\M_t\)_b \right\|_{\L\(H^s,H^s\)} \leqslant c \left\|f_t \right\|_{\infty} \Lambda^s \sup_{\textrm{supp } f_t} \left|F'\right|^{-\frac{1}{2}}
;\]
\item for all integers $0 \leqslant k \leqslant N$, all $s > k + \frac{3}{2}$ and all $s' < r - 1 $ the map
\[
t \mapsto \(M_t\)_c \in \L_{nuc} \(H^s, H^{s'} \)
\]
is $\mathcal{C}^k$ on $U$;
\item for all integers $0 \leqslant k \leqslant N$, all $ k + \frac{3}{2} < s < r-1$ and all $\epsilon > 0 $ the map 
\[
t \mapsto \(\M_t\)_b \in \L\(H^s, H^{s-k-\epsilon}\) 
\]
is $\mathcal{C}^k$ on $U$.
\end{enumerate}
\end{lm}

In this statement, $\L_{nuc} \(H^s, H^{s'} \)$ denotes the space of nuclear operators from $H^s$ to $H^{s'}$ equipped with the nuclear operator norm (see for instance \cite{Gohb} for definitions and main results).

In all this section, fix a real number $\Lambda >0$ and if $l,n$ are natural integers write
\begin{equation*}
l \hra n \textrm{ if } 2^n \leqslant \Lambda 2^{l+6}, \; l \nhra n \textrm{ otherwise }. 
\end{equation*} 
Then set
\[
\A = \set{ \(l,n\) \in \N^2 : l \hra n}
.\]

First, we investigate the "bounded" term $\(\M_t\)_b$.

\begin{lm}\label{prebor}
Let $r \in \N^*$ and $s \in \R^*_+$. There exists a constant $c > 0$ that does not depend of $\Lambda$ such that for any $\mathcal{C}^{r+1}$-diffeomorphism $F : \R \to \R$ and any compactly supported $\mathcal{C}^r$ function $f : \R \to \R$, setting for all $\phi \in H^s$
\[
\M \phi = f.\(\phi \circ F\)
,\]
then for any finite subset $E$ of $\A$ we have
\[
\left\|\sum_{ \(l,n\) \in E} \op \(\psi_n\) \M \op \(\psi_l\) \right\|_{H^s \to H^s} \leqslant c \|f\|_{\infty} \Lambda^s \sup_{\textup{ supp }f} |F'|^{-\frac{1}{2}} 
.\]
\end{lm}

\begin{proof}
Let $\phi \in \mathcal{S}$. For all $l \in \N$ set $\phi_l = \op\(\psi_l\)\phi$. Then 
\begin{align*}
\left\| \sum_{ \(l,n\) \in E} \op \(\psi_n\) \M \phi_l \right\|_{H^s}^2 & = \s{\R}{\(1+\xi^2\)^s \left| \sum_{n \geqslant 0} \psi_n\(\xi\) \(\sum_{l : \(l,n\) \in E} \widehat{\M \phi_l}\(\xi\) \)\right|^2}{\xi}
.\end{align*}
Since for all $\xi \in \R$ there are at most three values of $n$ for which $\psi_n\(\xi\) \neq 0$, Cauchy-Schwartz implies that
\begin{align*}
\left\| \sum_{ \(l,n\) \in E} \op \(\psi_n\) \M \phi_l \right\|_{H^s}^2 & \leqslant 3 \sum_{n \geqslant 0} \s{\R}{ \(1+ \xi^2\)^s \left|\psi_n\(\xi\) \(\sum_{l : \(l,n\) \in E} \widehat{\M \phi_l}\(\xi\) \)\right|^2}{\xi}  \\
   & \leqslant 3 \times 16^s  \sum_{n \geqslant 0} 4^{sn} \s{\R}{ \left|\psi_n\(\xi\) \(\sum_{l : \(l,n\) \in E} \widehat{\M \phi_l}\(\xi\) \)\right|^2}{\xi} \\
   & \leqslant 3 \times 16^s \s{\R}{\sum_{n \geqslant 0} \left| \sum_{l : \(l,n\) \in E} 2^{sn-sl} 2^{sl} \psi_n\(\xi\) \widehat{\M \phi_l}\(\xi\) \right|^2 }{\xi} 
.\end{align*}
Applying Cauchy-Schwartz again we get
\begin{align*}
  \left\| \sum_{ \(l,n\) \in E} \op \(\psi_n\) \M \phi_l \right\|_{H^s}^2  & \leqslant 3 \times 16^s \s{\R}{\sum_{n \geqslant 0} \(\sum_{l : \(l,n\) \in E} 2^{sn-sl} \) \( \sum_{l : \(l,n\) \in E} 2^{sn-sl} 4^{sl} \left| \psi_n\(\xi\) \widehat{\M \phi_l}\(\xi\)\right|^2 \) }{\xi} \\
  & \leqslant 3 \times 16^s \Lambda^s \frac{2^{6s}}{1-2^{-s}} \s{\R}{ \sum_{ \(l,n\) \in E} 2^{sn-sl} 4^{sl} \left| \psi_n\(\xi\) \widehat{\M \phi_l}\(\xi\)\right|^2  }{\xi} \\
   & \leqslant 3 \times 16^s \Lambda^s \frac{2^{6s}}{1-2^{-s}} \s{\R}{ \sum_{l \geqslant 0} 4^{sl} \(\sum_{ n : \(l,n\) \in E} 2^{\frac{s\(n-l\)}{2}} \left| \psi_n\(\xi\) \widehat{\M \phi_l}\(\xi\)\right|\)^2  }{\xi} \\
   & \leqslant 3 \times 16^s \Lambda^{2s} \frac{2^{12s}}{1-2^{-s}} \sum_{l \geqslant 0} 4^{sl} \s{\R}{ \left|\widehat{\M \phi_l}\(\xi\)\right|^2}{\xi} \\
   & \leqslant 3 \times 16^s \Lambda^{2s} \frac{2^{12s}}{1-2^{-s}} \sum_{l \geqslant 0} 4^{sl} \left\|\M \phi_l \right\|_{L^2}^2 \\
   & \leqslant 3 \times 16^s \Lambda^{2s} \frac{2^{12s}}{1-2^{-s}} \left\|f\right\|_{\infty}^2 \sup_{\textrm{ supp } f} \left|F'\right|^{-1} \sum_{l \geqslant 0} 4^{sl} \left\|\phi_l \right\|^2_{L^2} \\
   & \leqslant c^2 \Lambda^{2s} \left\| f \right\|_{\infty}^2 \sup_{\textrm{ supp } f } \left|F'\right|^{-1} \left\|\phi \right\|_{H^s}
.\end{align*}
We used \eqref{equivalent} in the last line.
\end{proof}

\begin{lm}\label{bor}
Under the hypotheses of Lemma \ref{prebor}, the series
\[
\sum_{l \hra n} \op\(\psi_n\) M \op\(\psi_l\)
\]
converges in $\L\(H^s,H^s\)$ equipped with the weak operator topology to an operator with operator norm bounded by 
\[
 c \|f\|_{\infty} \Lambda^s \sup_{\textup{supp }f} |F'|^{-\frac{1}{2}} 
.\]
\end{lm}

\begin{proof}
The net of the partial sums of this series is bounded according to Lemma \ref{prebor} and has a unique accumulation point since it converges for the operator norm topology on $\L\(H^1,H^{-1}\)$ for instance.
\end{proof}

Now, we study the "nuclear" term $\(\M_t\)_c$ of the decomposition given by lemma \ref{decloc}. One of its most important property is being regularizing.

\begin{lm}\label{nuc}
Let $r \in \N^*$ and $m \in \N$. Let $s,s'>0$ with $s > m + \frac{3}{2}$ and $s' < r - 1$. Let $f : \R \to \R$ be a $\mathcal{C}^r$ function compactly supported in a bounded open interval $\left]a,b\right[$ and let $F : \R \to \R$ be a $\mathcal{C}^{r+1}$-diffeomorphism with derivative bounded by $\Lambda$ on $\textrm{ supp } f$. Then there is a summable sequence $\(b_{n,l}\)_{l \nhra n}$ that only depends on $a,b,s,s',m, \Lambda$, the $\mathcal{C}^r$ norm of $f$, and the $\mathcal{C}^{r+1}$ norm of $F$ on $ \left[a,b\right]$, such that, setting for all $\phi \in \mathcal{S}$
\[
\M \phi = f \( \phi^{\(m\)} \circ F\)
,\]
for all $l \nhra n$ the operator $\op \(\psi_n\) \M \op\(\psi_l\)$ extends to a nuclear operator from $H^s$ to $H^{s'}$ with nuclear operator norm bounded by $b_{n,l}$.
\end{lm}

\begin{proof}
Since $\phi \mapsto \phi^{\(m\)}$ is bounded from $H^s$ to $H^{s-m}$ and commutes with $\op\(\psi_l\)$, one only has to deal with the case $m=0$. 

In the following, all the $C_i$ are constants depending only on $a,b,s,s',m, \Lambda$, the $\mathcal{C}^r$ norm of $f$ on $\left[a,b\right]$ and the $\mathcal{C}^{r+1}$ norm of $F$ on $F^{-1} \left[a,b\right]$.

Fix $\epsilon > 0$ such that $s > \frac{3}{2} + \epsilon$ and set $a : \eta \in \R \to \(1+\eta^2\)^{\frac{1+\epsilon}{2}}$.

Fix $l \nhra n$ with $l \neq 0$. If $\phi \in \mathcal{S}$ write for all $x \in \R$ 
\begin{equation}\label{noy}
\op\(\tilde{\psi_n}\) \M \op\(\psi_l \) \phi \(x\) = \frac{1}{\(2\pi\)^2} \s{\R}{V_{n,l}\(x,y\) \op\(a\)\phi\(y\)}{y} 
\end{equation}
where $V_{n,l}$ is defined by 
\[
\forall x,y \in \R : V_{n,l}\(x,y\) = \int_{\R^3} e^{i\(x-w\)\xi + i\(F\(w\) - y\) \eta } f\(w\) \tilde{\psi_n}\(\xi\) \frac{\psi_l\(\eta\)}{\(1+\eta^2\)^{\frac{1+\epsilon}{2}}} \mathrm{d} w \mathrm{d} \xi \mathrm{d} \eta
.\]
If $x$ doesn't lie in $\left[a,b\right]$, integrating by parts $r$ times, we get
\begin{equation}\label{eq1}
V_{n,l}\(x,y\) = i^{r} \int_{\R^3} \frac{e^{i\(x-w\)\xi + i\(F\(w\) - y\) \eta }}{\(x-w\)^{r} } f\(w\) \tilde{\psi_n}^{\(r\)}\(\xi\) \frac{\psi_l\(\eta\)}{\(1+\eta^2\)^{\frac{1+\epsilon}{2}}} \mathrm{d} w \mathrm{d} \xi \mathrm{d} \eta
\end{equation}
and thus, setting $K = \left[a,b\right]$ and recalling \eqref{decroit},
\begin{equation}\label{decr}
\left|V_{n,l}\(x,y\)\right| \leqslant  \frac{C_1}{d\(x,K\)^{r}} 2^{-\epsilon l} 2^{-\(r-1\)n} 
.\end{equation}
We notice
\begin{equation*}
\forall w \in \textrm{ supp } f : \forall \xi, \eta \in \R : \tilde{\psi_n}\(\xi\) \psi_l \(\eta\) \neq 0 \Rightarrow \left| \xi - F'\(w\) \eta \right| \geqslant 2^{n-4} 
.\end{equation*}
Then we write
\begin{align*}
V_{n,l}\(x,y\)  & = \int_{\R^3} i\(F'\(w\)\eta - \xi\)  e^{i\(x-w\)\xi + i\(F\(w\) - y\) \eta } \frac{f\(w\)}{i \(F'\(w\) \eta - \xi\) } \tilde{\psi_n}\(\xi\) \frac{\psi_l\(\eta\)}{\(1+\eta^2\)^{\frac{1+\epsilon}{2}}} \mathrm{d} w \mathrm{d} \xi \mathrm{d} \eta \\
    & = i  \int_{\R^3}  e^{i\(x-w\)\xi + i\(F\(w\) - y\) \eta } \drond{}{w} \( \frac{f\(w\)}{ \(F'\(w\) \eta - \xi\) }\) \tilde{\psi_n}\(\xi\) \frac{\psi_l\(\eta\)}{\(1+\eta^2\)^{\frac{1+\epsilon}{2}}} \mathrm{d} w \mathrm{d} \xi \mathrm{d} \eta
.\end{align*}
Using this trick $r$ times we write
\begin{equation}\label{eq2}
V_{n,l}\(x,y\) = \int_{\R^3} e^{i\(x-w\)\xi + i\(F\(w\) - y\) \eta } \tilde{\psi_n}\(\xi\) \frac{\psi_l\(\eta\)}{\(1+\eta^2\)^{\frac{1+\epsilon}{2}}} \Phi\(w,\xi,\eta\) \mathrm{d} w \mathrm{d} \xi \mathrm{d} \eta
,\end{equation}
with $\Phi$ bounded by $C_2 2^{-rn}$ when $\xi \in \textrm{supp} \(\tilde{\psi_n}\)$ and $\eta \in \textrm{supp}\(\tilde{\psi_l}\)$, and thus,
\begin{equation}\label{noz}
\left|V_{n,l}\(x,y\)\right| \leqslant C_3 2^{-\epsilon l} 2^{-\(r-1\)n} 
.\end{equation}

From \eqref{decr} and \eqref{noz}, we find for all $y \in \R$ :
\begin{equation}\label{osci}
\left\|V_{n,l}\( \cdot, y\)\right\|_{L^2} \leqslant C_4 2^{-\epsilon l} 2^{-\(r-1\)n}.
\end{equation}

If $y$ doesn't lie in $F\( \left[a,b\right]\)$, we may improve \eqref{osci} by integrating by parts twice on $\eta$ in \eqref{eq1} and \eqref{eq2}. This way we get
\begin{equation}\label{oscil}
\left\|V_{n,l}\( \cdot, y\)\right\|_{L^2} \leqslant C_4 2^{-\epsilon l} 2^{-\(r-1\)n} f\(y\)
\end{equation}
where $f$ is a positive integrable function that only depends on $a,b$, the $\mathcal{C}^r$ norm of $f$ on $\left[a,b\right]$, and the $\mathcal{C}^{r+1}$ norm of $F$ on $ \left[a,b\right]$.

Using the same kind of arguments, one can easily show that the function $y \in \R \mapsto V_{n,l}\(\cdot,y\) \in L^2$ is (Lipschitz-)continuous. Furthermore the function $y \in \R \mapsto \delta_y \circ \op\(a\) \in \(H^s\)'$ is (Hölder-)continuous. Thus the integral
\begin{equation*}
\frac{1}{\(2\pi\)^2}\s{\R}{\(\delta_y \circ \op\(a\)\) \otimes V_{n,l}\(.,y\)}{y} \in \L_{nuc}\( H^s,L^2\)
\end{equation*}
is well-defined and extends $\op\(\tilde{\psi_n}\) \M \op\(\psi_l\)$ according to \eqref{noy}. Moreover, its nuclear operator norm is bounded by
\[
\frac{1}{\(2 \pi\)^2}\s{\R}{\left\|\delta_y \circ \op\(a\) \right\|_{\(H^s\)'} \left\| V_{n,l}\(.,y\)\right\|_{L^2}}{y}  \leqslant C_5 2^{- \epsilon l} 2^{-\(r-1\)n}
.\]
Thus, the nuclear operator from $H^s$ to $H^{s'}$ defined by
\[
\frac{1}{\(2 \pi\)^2} \op \(\psi_n\) \circ \s{\R}{\(\delta_y \circ \op\(a\)\) \otimes V_{n,l}\(.,y\)}{y} 
\] 
extends the operator $\op \(\psi_n\) \M \op \(\psi_l\) = \op \(\psi_n\) \op \(\tilde{\psi_n}\) \M \op \( \psi_l\)$ and has nuclear operator norm bounded by $\left\|\op \(\psi_n\)\right\|_{L_2 \to H^{s'}} C_5 2^{- \epsilon l} 2^{-\(r-1\)n} \leqslant C_6 2^{- \epsilon l} 2^{-\(r-1-s'\)n} = b_{n,l}$ (recall \eqref{norme}).

\end{proof}

\begin{rmq}\label{weaken}
The constant $\frac{3}{2}$ that appears in Lemma \ref{nuc} must be understood as $1 + \frac{1}{2}$ and may be replaced by $1 + \frac{1}{p}$ by working with spaces $H_p^t \(\R\)$ as in \cite{Bal2} (here, $1$ is thought as the dimension and thus may be replaced by some integer $D$ to get a more general result). This is one way of weakening the hypotheses of Theorem \ref{main}. 
\end{rmq}

We immediately deduce the following result from Lemmas \ref{bor} and \ref{nuc}.

\begin{cor}\label{borne}
Let $r \in \N^*$ and $m \in \N$. Let $s \in \R$ with $m+ \frac{3}{2} < s < m + r - 1$. Let $f : \R \to \R$ be a $\mathcal{C}^r$ function compactly supported in a bounded open interval $\left]a,b\right[$ and $F : \R \to \R$ be a $\mathcal{C}^{r+1}$-diffeomorphism. Then setting for all $\phi \in \mathcal{S}$
\[
\M \phi = f \( \phi^{\(m\)} \circ F\)
,\]
the operator $\M$ is bounded from $H^s $ to $H^{s-m}$. Moreover its operator norm between these spaces is bounded by a constant that only depends of $a,b,s,m, \lambda$, the $\mathcal{C}^r$ norm of $f$, and the $\mathcal{C}^{r+1}$ norm of $F$ on $\left[a,b\right]$.
\end{cor}

Finally, we investigate the regularity of the dependence in $t$ of the operator $\M_t$.

\begin{lm}\label{cont}
Let $r \in \N^*$ and $m \in \N$. Let $s \in \R$ with $m+ \frac{3}{2} < s < m + r - 1$. Let $U$ be an open set of $\R^D$ for some integers $D$. Let $t \mapsto f_t \in \mathcal{C}^r\(\R\)$ be a continuous function on $U$ whose values are supported in a bounded open interval $\left]a,b\right[$. Let $t \mapsto F_t \in \mathcal{C}^{r+1} \(\R\)$ be a continuous function on $U$ whose values are diffeomorphisms of $\R$ into itself.  For all $\phi \in \mathcal{S}$ and $t \in U$ set
\[
\M_t \phi = f_t.\(\phi^{\(m\)} \circ F_t\) 
.\]
Then for all $\epsilon > 0 $ the function $t\mapsto \M_t \in \L\(H^s, H^{s-m-\epsilon}\)$ is continuous on $U$. 
\end{lm}

\begin{proof}
As previously, we may and will suppose that $m=0$. If $V$ is an open set of $\R$, we denote by $H_0^s\(V\)$ the closure in $H^s$ of the set of $\mathcal{C}^{\infty}$ functions compactly supported in $V$. Then, we can write
\[
\M_t = i \circ \M_t \circ \theta
\]
where in the left-hand side $\M_t$ is seen as an operator from $H^s$ ro $H^{s-\epsilon}$ and in the right-hand side $\M_t$ is seen as an operator from $H^s_0\(\left]a,b\right[\)$ to $H^s_0\(F^{-1}\left]a,b\right[\)$, $i$ is the (compact) embedding of $H^s_0\(F^{-1}\left]a,b\right[\)$ in $H^{s-\epsilon}$ and  $\theta : H^s \to H^s_0\(\left]a,b\right[\)$ is the operator of multiplication by a $\mathcal{C}^{\infty}$ function supported in $\left] a,b\right[$ with value $1$ on the support of the $f_t$ Thanks to the compacity of $i$, we only need to show that $t \mapsto \M_t \in \L\(H^s_0\(\left]a,b\right[\), H^s_0\(F^{-1}\left]a,b\right[\)\)$ is continuous for the strong topology, that is for all $\phi \in H^s_0\(\left]a,b\right[\)$ the function $t \mapsto M_t \phi \in H^s_0\(F^{-1}\left]a,b\right[\)$ is continuous. Thanks to Corollary \ref{borne}, we only need to prove this when $\phi$ is $\mathcal{C}^{\infty}$ function supported in $\left]a,b\right[$. But for such a function, the continuity is obvious (it even holds in $\mathcal{C}^r_0\(F^{-1}\left]a,b\right[ \)$).
\end{proof}

\begin{lm}\label{der}
Let $r \in \N^*, m \in \N$ and $N \in \N$. Let $s \in \R$ with $N+m+\frac{3}{2} < s < m + r -1$. Let $U$ be an open set of $\R^D$ for some integer $D$. Let $t \mapsto F_t \in \mathcal{C}^{r+1} \(\R\)$ be a $\mathcal{C}^N$ function on $U$ whose values are diffeomorphisms. Let $t \mapsto f_t \in \mathcal{C}^r\(\R\)$ be a $\mathcal{C}^N$ function on $U$ whose values are supported in a bounded open interval $\left]a,b\right[$. For all $\phi \in \mathcal{S}$ and $t \in U$ set
\[
\M_t \phi = f_t.\(\phi^{\(m\)} \circ F_t\) 
.\]
Then for all $\epsilon > 0 $ the function $t \in U \mapsto \M_t \in \L\(H^s, H^{s-m-N-\epsilon}\)$ is $\mathcal{C}^N$. 
\end{lm}

\begin{proof}
The proof is an induction on $N$. The case $N=0$ has been dealt with in Lemma \ref{cont}. Let $N \geqslant 1$. As usual, we may suppose $m=0$. Furthermore, we only need to deal with the case $\epsilon < s -N - \frac{1}{2}$. Let $i \in \set{1,\dots,D}$ and for all $t = \(t_1,\dots,t_D\) \in U$ and $\phi \in H^s$, write
\[
A_t = \drond{f_t}{t_i} . \(\phi \circ F_t\) + f_t . \drond{F_t}{t_i} \(\phi' \circ F_t\)
.\]
By induction hypothesis, $t \mapsto A_t  \in \L\(H^s, H^{s-m-N-\epsilon}\) $ is $\mathcal{C}^{N-1}$. Consequently, we only need to check that $A_t$ is the partial derivative of $t \mapsto \M_t$ with respect to $t_i$. Let $T = \(T_1,\dots,T_d\) \in U$ and set for $t_i$ sufficiently close to $T_i$ :
\[
F\(t_i\) = \M_T + \int_{T_i}^{t_i} A_{t_1,\dots,\tau_i,\dots,t_d} \mathrm{d} \tau_i \in \L\(H^s, H^{s-m-N-\epsilon}\)
.\]
Now, if $\phi \in \mathcal{S}$ and $y \in \R$, since $\delta_y$ is continuous on $H^{s-N-\epsilon}$, we have
\[
F\(t_i\) \phi \(y\) = \M_T \phi \(y\) + \int_{T_i}^{t_i} A_{t_1,\dots,\tau_i,\dots,t_D}\phi\(y\) \mathrm{d} \tau_i \in \L\(H^s, H^{s-m-N-\epsilon}\) = \M_{T_1,\dots,t_i,\dots,T_D}\phi\(y\)
.\]
Thus $F\(t_i\) \phi = \M_{T_1,\dots,t_i,\dots,T_D}\phi$ and finally $F\(t_i\) =\M_{T_1,\dots,t_i,\dots,T_D}$. Consequently, $A_t$ is the derivative of $t \mapsto \M_t$ with respect to $t_i$.
\end{proof}

Now, we can prove our "local decomposition" lemma.

\begin{proof}[Proof of Lemma \ref{decloc}]
Set 
\[
\(\M_t\)_b = \sum_{l \hra n} \op \(\psi_n\) \M_t \op \(\psi_l\) \textrm{ and } \(\M_t\)_c = \sum_{l \nhra n} \op \(\psi_n\) \M_t \op \(\psi_l\)
.\]
Then the first point is a consequence of Lemma \ref{bor}. The second point is deduced from Lemma \ref{der} by a standard argument of dominated convergence (the domination being a consequence of Lemma \ref{nuc}). Finally, the third point is an immediate consequence of the second point and Lemma \ref{der}.
\end{proof}

\section{Decomposition of the transfer operator}\label{3}

Let $U$ be an open set of $\R^D$. Let $r \in \N^{*}$, $N \in \N$. Let $0 < \lambda <1$. Let $t \mapsto T_t \in \mathcal{C}^{r+1}\(S^1,S^1\)$ be a $\mathcal{C}^N$ function on $U$ whose values are expanding maps of the circle with expansion constant $\lambda^{-1}$ that is
\begin{equation} \label{dilatation}
\forall t \in U : \forall x \in S^1 : \left|T_t'\(x\)\right| \geqslant \lambda^{-1}
.\end{equation}
Let $t \mapsto g_t \in \mathcal{C}^r\(S^1\)$ be a $\mathcal{C}^N$ function on $U$ (in the application, we shall choose $D=2$, $t=\(\tau,u\)$, $T_t = T_\tau$ and $g_t = -u g -\log \left|T_\tau'\right|$). Our main object of study is the transfer operator defined for all $t \in U$ and all $\phi \in \C^{S^1}$ by
\[
\mathcal{L}_t \phi : x \mapsto \sum_{y \in T_t^{-1}\(\set{x}\)} e^{g_t\(y\)} \phi\(y\) 
.\]
We shall associate to $\mathcal{L}_t$ an operator $\mathcal{K}_t$ with similar properties. Then we shall apply Lemma 2.8 to get a similar decomposition for the operator $\mathcal{K}_t$. The properties of this decomposition are stated in Proposition \ref{fonda}.

We need further notation to do so. Let $K$ be a compact subset of $U$ and $\tilde{K}$ be a compact neighbourhood of $K$ in $U$. We may choose a finite cover $\alpha = \(V_{\omega}\)_{\omega \in \Omega}$ of $S^1$ by open intervals with the following properties :
\begin{enumerate}[label=\arabic*.]
\item for all $\omega \in \Omega$, the canonical projection $\pi
: \R \to S^1$ has a $\mathcal{C}^{\infty}$ local inverse $\kappa_\omega$ defined on $V_\omega$;
\item for all $t \in \tilde{K}$ and all $\omega \in \Omega$, the map $T_t$ induces a diffeomorphism from a neighbourhood of $\bar{V_\omega}$ to a neighbourhood of $T_t\(\bar{V_\omega}\)$;
\item for all $t \in \tilde{K}$ 
\begin{equation*}
\sup \set{\textrm{diam } V : V \in \bigwedge_{i=0}^{m-1} T_t^{-i} \alpha  } \underset{m \to + \infty}{\to} 0
;\end{equation*}
\item for all $t \in \tilde{K}$ and $m \in \N^*$ the elements of $\bigwedge_{i=0}^{m-1} T_t^{-i} \alpha $ are open intervals;
\item denoting for all $\omega \in \Omega$ by $W_\omega$ the open interval of $S^1$ with same center than $V_\omega$ but three times as long, the cover $\tilde{\alpha} = \(W_\omega\)_{\omega \in \Omega}$ also satisfies the four properties above.
\end{enumerate}
Indeed, these properties hold as soon as the diameter of the elements of $\alpha$ is small enough, "small enough" being uniform in $t$ thanks to the compacity of $\tilde{K}$.

For all $m \in \N^*$, $t \in U$ and $\overrightarrow{\omega} = \(\omega_0,\dots, \omega_{m-1}\) \in \Omega^m$ let us write $ V_{\overrightarrow{\omega},t} = \bigcap_{i=0}^{m-1} T_t^{-i} V_{\omega_i} $ and $ \alpha_{m,t} = \bigwedge_{i=0}^{n-1} T_t^{-i} \alpha = \set{V_{\overrightarrow{\omega},t} : \overrightarrow{\omega} \in \Omega^m}$. Replacing, $V_\omega$ by $W_\omega$, we define in the same way $W_{\overrightarrow{\omega},t}$ and $\tilde{\alpha}_{m,t}$.  
For all $m \in \N^*$, $t \in U$ and $x \in S^1$ write
\begin{equation}\label{gmt}
g_{m,t}\(x\) = \sum_{i=0}^{m-1} g_t\(T_t^i\(x\)\)
.\end{equation}
By a standard bounded distortion argument, there is a constant $M>0$ such that for all $t \in \tilde{K}$, $m \in \N^*$ and $V \in \tilde{\alpha}_{m,t}$, if $x,y \in V$ then
\begin{equation}\label{distortion}
\left|g_{m,t}\(x\) - g_{m,t}\(y\)\right| \leqslant M \textrm{ and } \left|\(T_t^m\)'\(x\)\right|  \leqslant M \left|\(T_t^m\)'\(y\)\right| 
.\end{equation}

Now, choose a $\mathcal{C}^{\infty}$ partition of unity $\(\theta_\omega\)_{\omega \in \Omega}$ adapted to the cover $\alpha$. For all $\omega \in \Omega$, choose a $\mathcal{C}^{\infty}$ function $h_\omega$ compactly supported in $\kappa_\omega \(V_\omega\)$ with $h_\omega = 1$ on $\kappa_\omega\(\textrm{supp } \theta_\omega\)$. For all $s \in \R$ set 
\[
\B^s = \bigoplus_{\omega \in \Omega} H^s
\]
equipped with the norm
\[
\|.\|_{\B^s} : \(\phi_\omega\)_{\omega \in \Omega} \mapsto \sqrt{\sum_{\omega \in \Omega} \left\|\phi_{\omega}\right\|_{H^s}^2}
,\]
which ensures that $\B^s$ is a Hilbert space.

Define
\[
S : \(\phi_\omega\)_{\omega \in \Omega} \in \(\C^{\R}\)^{\Omega} \mapsto \sum_{\omega \in \Omega}  \theta_{\omega} \(\phi_\omega \circ \kappa_\omega\) \in \C^{S^1}
\]
and
\[
P : \phi \in \C^{S^1} \mapsto \(h_\omega \phi \circ \pi \)_{\omega \in \Omega} \in \(\C^{\R}\)^{\Omega}
,\]
and notice that $S \circ P = Id$.

For all $t \in U$ define $\mathcal{K}_t$ on $\bigoplus_{\omega \in \Omega} \mathcal{S}$ by
\[
 \mathcal{K}_t \phi = \(P \circ \L_t \circ S\) \phi
.\]

\begin{prop}\label{fonda}
For all $t \in K$ and all $\frac{3}{2} < s < r-1$ the operator $\mathcal{K}_t$ extends to a bounded operator from $\B^s$ to itself. Moreover, for all $m \in \N^*$ the operator $K_t^m$ can be written as a sum \footnote{This decomposition does not depend of $s$ in the following sense: the operators $\(\mathcal{K}_t^m\)_b $ and $\(\mathcal{K}_t^m\)_c $ commute with the natural injections between spaces $\B^s$ for different values of $s$.}
\begin{equation}\label{decomp}
\mathcal{K}_t^m = \(\mathcal{K}_t^m\)_b + \(\mathcal{K}_t^m\)_c
\end{equation}
such that the following properties hold:
\begin{enumerate}[label=\roman*.]
\item for all $\frac{1}{2} < s < r$, there exists a constant $c>0$ such that for all $m \in \N^*$ and all $t \in K$:
\begin{equation}\label{bounded}
\left\|\(\mathcal{K}_t^m\)_b\right\|_{\L\(\B^s, \B^s\)} \leqslant c \lambda^{m\(s-\frac{1}{2}\)} \inf_{\beta \textup{ subcover of } \alpha_{m,t}} \sum_{V \in \beta} \exp \( \sup_{V} g_{m,t}\);
\end{equation}
\item for all integers $ 0 \leqslant k \leqslant N$, all $\frac{3}{2} + k < s < r-1$, all $m \in \N^*$ and all $\epsilon > 0$ the map
\[
t \mapsto \(\mathcal{K}_t^m\)_b \in \L\(\B^s, \B^{s-k-\epsilon}\)
\] 
extends to a $\mathcal{C}^k$ function on a neighbourhood of $K$;
\item for all integers $0 \leqslant k \leqslant N$ , all $s > k + \frac{3}{2}$, all $s' < r - 1$ and all $m \in \N^*$ the map
\[
t \mapsto \(\mathcal{K}_t^m\)_c \in \L_{nuc}\(\B^s,\B^{s'}\)
\]
extends to a $\mathcal{C}^k$ function on a neighbourhood of $K$.

\end{enumerate}
\end{prop}

\begin{proof}
The boundedness of $\mathcal{K}_t$ on $\B^s$ will be a consequence of the decomposition \eqref{decomp} for $m=1$. Fix an integer $m \in \N^*$.

If $t \in K$ choose a subset $I=I_{t,m}$ of $\Omega^m$ that reaches the infimum
\begin{equation}\label{choix}
\sum_{\overrightarrow{\omega} \in I} \exp\( \sup_{V_{\overrightarrow{\omega},t}} g_{m,t} \) = \inf_{J \subseteq \Omega^m, S^1 = \bigcup_{\overrightarrow{\omega} \in J} W_{\overrightarrow{\omega},t}} \sum_{\overrightarrow{\omega} \in J} \exp\( \sup_{V_{\overrightarrow{\omega},t}} g_{m,t} \)
.\end{equation}
Then find a neighbourhood $U_{t,m} \subseteq \tilde{K}$ of $t$ in $U$ with the following properties:
\begin{enumerate}[label=\arabic*.]
\item for all $t' \in U_{t,m}$
\begin{equation}\label{recouvrement}
S^1 = \bigcup_{\overrightarrow{\omega} \in I} W_{\overrightarrow{\omega},t'};
\end{equation}
\item for all $t' \in U_{t,m}$
\begin{equation}\label{ruse}
\sum_{\overrightarrow{\omega} \in I} \exp\( \sup_{V_{\overrightarrow{\omega},t'}} g_{m,t'} \) \leqslant 2 \inf_{J \subseteq \Omega^m,S^1 = \bigcup_{\overrightarrow{\omega} \in J} V_{\overrightarrow{\omega},t'} } \sum_{\overrightarrow{\omega} \in J} \exp\( \sup_{V_{\overrightarrow{\omega},t'}} g_{m,t'} \) ;
\end{equation}
\item for all $t' \in U_{t,m}$ and $\overrightarrow{\omega} \in I$ 
\begin{equation}\label{maitrise}
\frac{1}{2} \inf_{W_{\overrightarrow{\omega},t}} \left|\(T_t^m\)'\right| \leqslant \inf_{W_{\overrightarrow{\omega},t'}} \left|\(T_{t'}^m\)'\right|\leqslant 2 \inf_{W_{\overrightarrow{\omega},t}} \left|\(T_t^m\)'\right|
.\end{equation}
\end{enumerate}
We explain briefly how to find such a neighbourhood.  The first point is easy, one only needs to notice that \eqref{recouvrement} is equivalent to
\[
\forall x \in S^1 : \sum_{\(\omega_0,\dots,\omega_{m-1}\) \in I} \prod_{i=0}^{m-1} d\(T_t^i\(x\), S^1 \setminus W_{\omega_i}\) > 0
.\]
The third point is an argument of continuity, using the fact that for all $t' \in \tilde{K}$ the set $W_{\overrightarrow{\omega},t'}$ is an interval. The second point is more complicated. If $J \subseteq \Omega^m$ is such that $\(W_{\overrightarrow{\omega},t}\)_{\overrightarrow{\omega} \in J}$ doesn't cover $S^1$ then for $t'$ sufficiently close to $t$, $\(V_{\overrightarrow{\omega},t'}\)_{\overrightarrow{\omega} \in J}$ doesn't cover $S^1$. Consequently, for $t'$ sufficiently close to $t$ we have
\[
\inf_{J \subseteq \Omega^m,S^1 = \bigcup_{\overrightarrow{\omega} \in J} W_{\overrightarrow{\omega},t} } \sum_{\overrightarrow{\omega} \in J} \exp\( \sup_{V_{\overrightarrow{\omega},t'}} g_{m,t'} \)  \leqslant \inf_{J \subseteq \Omega^m,S^1 = \bigcup_{\overrightarrow{\omega} \in J} V_{\overrightarrow{\omega},t'} } \sum_{\overrightarrow{\omega} \in J} \exp\( \sup_{V_{\overrightarrow{\omega},t'}} g_{m,t'} \) 
.\]
But the infimum on the left-hand side of this inequality is taken on a set that does not depend of $t'$, so we can use the same kind of argument as for the third point, recalling \eqref{choix}.

Now, $\(U_{t,m}\)_{t \in K}$ is an open cover of $K$ and consequently, one only needs to get the decomposition \eqref{decomp} on each of its elements separately (then glue the different decompositions using a partition of unity). So fix $t_0 \in K$ and write $I=I_{t_0,m}$ the subset of $\Omega^m$ that appears in the definition of $U_{t_0,m}$. For all $\omega \in \Omega$ choose a $\mathcal{C}^{\infty}$ function $\tilde{\chi}_\omega : S^1 \to \R$ such that $0 \leqslant \tilde{\chi}_\omega \leqslant 1$ and $\tilde{\chi}_{\omega}\(x\) > 0$ if and only if $x \in W_\omega$. Then for all $m \in \N^*$, $t \in U_{t_0,m}$ and $\overrightarrow{\omega} = \(\omega_0,\dots,\omega_{m-1}\) \in I$, set:
\[
\tilde{\chi}_{\overrightarrow{\omega},t} : x \in S^1 \mapsto \prod_{i=0}^{m-1} \tilde{\chi}_{\omega_i} \(T_t^i\(x\)\)
\]
and
\[
\chi_{\overrightarrow{\omega},t} : x \in S^1 \mapsto \frac{\tilde{\chi}_{\overrightarrow{\omega},t}\(x\)}{\sum_{\overrightarrow{\omega}' \in I} \tilde{\chi}_{\overrightarrow{\omega}',t}\(x\)} 
,\]
which is well-defined thanks to \eqref{recouvrement}.
Thus we have for all $t \in U_{t_0,m}$:
\begin{equation*}
\sum_{\overrightarrow{\omega} \in I} \chi_{\overrightarrow{\omega},t} = 1 \textrm{ and } \forall \overrightarrow{\omega} \in I : \forall x \in S^1 : \chi_{\overrightarrow{\omega},t}\(x\) > 0 \Leftrightarrow x \in W_{\overrightarrow{\omega},t}  
.\end{equation*}
Then for all $\overrightarrow{\omega} \in I$, $t \in U_{t_0,m}$ and $\phi \in \mathcal{C}^{\infty}\(S^1\)$, define
\begin{equation*}
\L^m_{\overrightarrow{\omega},t} \phi : x \in S^1 \mapsto \sum_{T_t^m\(y\) = x} \chi_{\overrightarrow{\omega},t}\(y\) e^{g_{m,t}\(y\)} \phi\(y\) = \(\chi_{\overrightarrow{\omega},t} e^{g_{m,t}} \phi\) \circ \(\left. T_t^m \right|_{W_{\overrightarrow{\omega},t}}\)^{-1}\(x\)
\end{equation*}
and then 
\begin{equation*}
\mathcal{K}^m_{\overrightarrow{\omega},t} = D \circ  \L^m_{\overrightarrow{\omega},t} \circ S
.\end{equation*}
These definitions immediately imply for all $t \in U_{t_0,m}$
\begin{equation*}
\L^m_t = \sum_{\overrightarrow{\omega} \in I} \L^m_{\overrightarrow{\omega},t} \textrm{ and } \mathcal{K}^m_t = \sum_{\overrightarrow{\omega} \in I} \mathcal{K}^m_{\overrightarrow{\omega},t}
.\end{equation*}

Now, fix $\overrightarrow{\omega} \in I$ and write $\mathcal{K}^m_{\overrightarrow{\omega},t}$ as a matrix of operators $\(A_{\omega,\omega',t}\)_{\omega,\omega' \in \Omega}$ that is, for all $\phi = \(\phi_\omega\)_{\omega \in \Omega} \in \bigoplus_{\omega \in \Omega} \mathcal{S}$, we have
\[
\mathcal{K}^m_{\overrightarrow{\omega},t}\phi = \( \sum_{\omega' \in \omega} A_{\omega,\omega',t}\phi_{\omega'}\)_{\omega \in \Omega}
.\]

If $\omega,\omega' \in \Omega$ and $\phi \in \mathcal{S}$, we have
\begin{equation}\label{loc}
A_{\omega,\omega',t}\phi = h_\omega \( \chi_{\overrightarrow{\omega},t} e^{g_{m,t}} \theta_{\omega'} \) \circ \(\left. T_t^m \right|_{W_{\overrightarrow{\omega},t}}\)^{-1} \circ \pi . \(\phi \circ \kappa_{\omega'} \circ \(\left. T_t^m \right|_{W_{\overrightarrow{\omega},t}}\)^{-1} \circ \pi\)
.\end{equation}
The map $\kappa_{\omega'} \circ \(\left. T_t^m \right|_{W_{\overrightarrow{\omega},t}}\)^{-1} \circ \pi $ may be extended in a $\mathcal{C}^{r+1}$-diffeomorphism $F_t$ of $\R$ in a consistent way: take the inverse of a lift of $T_t^m$ that extends $\kappa_\omega \circ \(\left. T_t^m \right|_{W_{\overrightarrow{\omega},t}}\) \circ \left. \pi \right|_{\kappa_{\omega'} \(V_{\omega'}\)}$. Now apply Lemma \ref{decloc} with this $F_t$, $f_t = h_\omega \( \chi_{\overrightarrow{\omega},t} e^{g_{m,t}} \theta_{\omega'} \) \circ \pi \circ F_t $ and $\Lambda = 2 \(\inf_{W_{\overrightarrow{\omega},t_0}} \left|\(T_{t_0}^m\)'\right|\)^{-1}$, recalling \eqref{maitrise}. Thus, using Lemma \ref{decloc}, we get a decomposition of $A_{\omega,\omega',t}$ in
\[
A_{\omega,\omega',t} = \(A_{\omega,\omega',t}\)_b + \(A_{\omega,\omega',t}\)_c
\]
with the expected regularity and, for all $\frac{1}{2} < s < r$, a constant $c_s$, that only depends of $s$, such that
\begin{align*}
\left\|\(A_{\omega,\omega',t}\)_b\right\|_{H^s \to H^s} &  \leqslant c_s \exp\( \sup_{W_{\overrightarrow{\omega},t}} g_{m,t}\) \Lambda^s  \sup_{W_{\overrightarrow{\omega},t}} \left|\(T_t^m\)'\right|^{\frac{1}{2}} \\
     & \leqslant  4^s c_s e^M \exp\( \sup_{V_{\overrightarrow{\omega},t}} g_{m,t}\) \( \inf_{W_{\overrightarrow{\omega},t}} \left|\(T_t^m\)'\right| \)^{-s} \sup_{W_{\overrightarrow{\omega},t}} \left|\(T_t^m\)'\right|^{\frac{1}{2}} \\
     & \leqslant 4^s c_s e^M \sqrt{M}  \exp\( \sup_{V_{\overrightarrow{\omega},t}} g_{m,t}\) \( \inf_{W_{\overrightarrow{\omega},t}} \left|\(T_t^m\)'\right| \)^{-s+\frac{1}{2}} \\
     & \leqslant 4^s c_s e^M \sqrt{M} \exp\( \sup_{V_{\overrightarrow{\omega},t}} g_{m,t}\) \lambda^{m\(s-\frac{1}{2}\)}
,\end{align*}
where $M$ has been introduced in \eqref{distortion} (that we used on second and third line). We also used \eqref{maitrise} on second line and \eqref{dilatation} on the last line. From this, we deduce a decomposition of $\mathcal{K}^m_{\overrightarrow{\omega},t}$ with the expected regularity and the same estimate of the operator norm up to a multiplicative factor $\(\# \Omega\)^2$. Summing over $\overrightarrow{\omega} \in I$, we get the decomposition \eqref{decomp} with the expected regularity and 
\begin{align*}
\left\|\(\mathcal{K}^m_t\)_b\right\|_{H^s \to H^s} & \leqslant 4^s c_s e^M \sqrt{M} \(\# \Omega\)^2 \sum_{\overrightarrow{\omega} \in I} \exp\( \sup_{V_{\overrightarrow{\omega},t}} g_{m,t}\) \lambda^{m\(s-\frac{1}{2}\)} \\
     & \leqslant 2 \times 4^s c_s e^M \sqrt{M} \(\# \Omega\)^2 \lambda^{m\(s-\frac{1}{2}\)} \inf_{\beta \textrm{ subcover of } \alpha_{m,t}} \sum_{V \in \beta} \exp \( \sup_{V} g_{m,t}\)
\end{align*} 
thanks to \eqref{ruse}.

\end{proof}

\begin{rmq}\label{transder}
As a consequence of Theorem \ref{fonda} for all integers $0 \leqslant k \leqslant N$, for all $\frac{3}{2} < s < r-1$ , all $m \in \N^*$ and all $\epsilon$ the map
\[
t \mapsto \mathcal{K}_t \in \L\(\B^s, \B^{s-k-\epsilon}\)
\]
is $\mathcal{C}^k$ on a neighborhood of $K$.
\end{rmq}

\section{Flat trace of the transfer operator}\label{4}

This section is dedicated to the definitions and basic properties of the "flat trace" and "flat determinant", which are key tools in the study of dynamical determinants in \cite{Bal2} or \cite{Tsu}. Although we will not define these objects in the same way, it could be shown that both definitions agree in most cases.

For all $\epsilon >0$ we set for all $x \in \R$:
\begin{equation*}
\rho_\epsilon \(x\) = \frac{1}{\epsilon} \chi\(\frac{x}{\epsilon}\) \textrm{ and } \chi_\epsilon = \F^{-1}\(\rho_{\epsilon}\),
\end{equation*}
where $\F^{-1}$ is the inverse of the Fourier transform and $\rho : \R \to \R$ is a $\mathcal{C}^{\infty}$ function, taking values in $\left[0,1\right]$, compactly suported, of integral $1$ and identically equals to $1$ one a neighbourhood of $0$.

For all $\phi \in \mathcal{S}'$ we write
\begin{equation*}
J_\epsilon \phi = \rho_\epsilon \ast \(\chi_\epsilon \phi\)
\end{equation*}
and then for all $\phi = \(\phi_{\omega}\)_{\omega \in \Omega} \in \bigoplus_{\omega \in \Omega} \mathcal{S}'$
\begin{equation*}
I_\epsilon \phi = \(J_\epsilon \phi_{\omega}\)_{\omega \in \Omega}
.\end{equation*}

The basic properties of these operators are listed in the following lemma.

\begin{lm}\label{convolution}
For all $s,s' > \frac{1}{2}$ the following properties hold:
\begin{enumerate}[label=\roman*.]
\item for all $\epsilon >0$ , $J_\epsilon$ (resp. $I_\epsilon$) defines a nuclear operator of order $0$ from $H^s$ to $H^{s'}$ (resp. from $\B^s$ to $\B^{s'}$);
\item there is a constant $C$ such that for all $\epsilon \in \left]0,1\right]$ we have $\left\|J_\epsilon\right\|_{H^s \to H^{s}} \leqslant C$ (resp. $\left\|I_\epsilon\right\|_{\B^s \to \B^{s}} \leqslant C$);
\item for all $\phi \in H^s$ (resp. $\B^s$) , $J_\epsilon \phi$ (resp. $I_\epsilon \phi$) tends to $\phi$ in $H^s$ (resp. $\B^s$) as $\epsilon$ tends to $0$.
\end{enumerate}
\end{lm}

\begin{proof}
The first point is immediate since $J_\epsilon$ factorizes through $\mathcal{S}$ which is a nuclear space (see for instance the second part of \cite{Groth}).

The norm in $\L\(H^s,H^s\)$ of $\phi \in H^s \to \rho_\epsilon \ast \phi \in H^s$ is bounded by $\left\|\widehat{\rho_\epsilon}\right\|_{L^\infty} \leqslant \left\|\rho_\epsilon \right\|_{L^1}=1$.  We get a uniform bound on the norm in $\L\(H^s,H^s\)$ of $\phi \in H^s \mapsto \chi_\epsilon \phi$ by a classical Leibniz inequality (for instance corollary 4.2.2 of \cite{Trieb}).

The third point is a consequence of the second one and the fact that the convergence holds for function in $\mathcal{S}$ by an argument of density.
\end{proof}

Thus if $s \in \R$ and $A$ is a bounded operator from $H^s$ to itself, $A \circ J_\epsilon$ is a nuclear operator and since the Hilbert $H^s$ as the approximation property, we can set
\begin{equation*}
\textrm{tr}_\epsilon A = \textrm{tr} \( A \circ J_\epsilon \)
,\end{equation*}
and then the "flat trace" of $A$ is defined as
\begin{equation*}
\textrm{tr}^{\flat} A = \lim_{\epsilon \to 0} \textrm{tr}_\epsilon A
,\end{equation*}
provided the limit exists. Replacing $J_\epsilon$ by $I_\epsilon$ we get similar definitions for operators from $\B^s$ to itself. If we write a bounded operator $A$ on $\B^s$ as a matrix $\(A_{\omega, \omega'}\)_{\omega,\omega' \in \Omega}$ of bounded operators on $H^s$, it can easily be shown that
\begin{equation*}
\textrm{tr}_\epsilon A = \sum_{\omega \in \Omega} \textrm{tr}_\epsilon A_{\omega,\omega}.
\end{equation*}
Thus the the flat trace of $A$ is defined if and only if the flat trace of $A_{\omega,\omega}$ is defined for all $\omega \in \Omega$, if so 
\begin{equation*}
\textrm{tr}^{\flat} A = \sum_{\omega \in \Omega} \textrm{tr}^{\flat} A_{\omega,\omega}.
\end{equation*}
From the third point of Lemma \ref{convolution}, the "flat trace" coincides with the usual trace for nuclear operators.

Now, if $A : \B^s \to \B^s$ is such that for all $m \in \N^*$ the flat trace of $A^m$ is defined, set
\begin{equation}\label{flatdet}
\textrm{det}^{\flat}\(I - z A\) = \exp\(-\sum_{n \geqslant 1} \frac{\textrm{tr}^{\flat}\(A^n\)}{n}z^n\) = \sum_{n \geqslant 0} a_n z^n \in \C\left[\left[z\right]\right]
\end{equation} 
that is $\sum_{n \geqslant 0} a_n z^n$ is the formal power series recursively defined by
\begin{equation}\label{coeffdet}
a_0 = 1 \textrm{ and } a_n = -\frac{1}{n} \sum_{k=0}^{n-1} a_k \textrm{tr}^{\flat}\(A^{n-k}\) \textrm{ for } n \geqslant 1
.\end{equation}
Thus if $A$ and $B$ are operators such that $AB=BA= 0$ and for all $m \in \N^*$ the flat traces of $A^m$ and $B^m$ are defined, we have
\begin{equation*}
\textrm{det}^{\flat}\(I - z \(A+B\)\) = \textrm{det}^{\flat}\(I - z A\) \textrm{det}^{\flat}\(I - z B\)
.\end{equation*}

First, we show that all the powers of the transfer operator $\mathcal{K}_t$ have a flat trace.

\begin{lm}\label{ltrace}
For all $m \in \N^*$, all $t \in U$ and all $\frac{3}{2} < s < r-1$ the flat trace of the operator $\mathcal{K}_t^m : \B^s \to \B^s$ is defined and
\begin{equation}\label{trace}
\textrm{tr}^{\flat}\(\mathcal{K}_t^m\) = \sum_{T_t^m x = x} \frac{\exp\( g_{m,t}\(x\)\)}{1 - \(\(T_t^m\)'\(x\)\)^{-1}}.
\end{equation}
\end{lm}

\begin{proof}
As in the proof of Proposition \ref{fonda}, write
\begin{equation*}
\mathcal{K}_t^m = \sum_{\overrightarrow{\omega} \in I} \mathcal{K}^m_{\overrightarrow{\omega},t}
.\end{equation*}
Choose $\overrightarrow{\omega} \in I$ and write $\mathcal{K}_{\overrightarrow{\omega},t}^m$ as a matrix of operators $\(A_{\omega,\omega',t}\)_{\omega,\omega' \in \Omega}$. For all $\omega \in \Omega$, $\phi \in H^s$ and $x \in \R$ we can write thanks to \eqref{loc}
\[
\(A_{\omega,\omega,t} \circ J_\epsilon\)\phi \(x\) = \s{\R}{h_\omega\(x\)\(\chi_{\overrightarrow{\omega},t} e^{g_{m,t}} \theta_\omega\) \circ \pi \circ F_t\(x\) \rho_\epsilon\(F_t\(x\)-y\) \chi_\epsilon\(y\) \phi\(y\)}{y}
,\]
which expresses $A_{\omega,\omega,t}$ as an integral of nuclear operators.
Thus
\[
\textrm{tr}_\epsilon A_{\omega,\omega,t} = \s{\R}{h_\omega\(x\)\(\chi_{\overrightarrow{\omega},t} e^{g_{m,t}} \theta_\omega\) \circ \pi \circ F_t\(x\) \rho_\epsilon\(F_t\(x\)-x\) \chi_\epsilon\(x\)}{x}
.\]
Since $F_t$ has its derivative bounded by $\lambda^m <1$, the map $x \mapsto F_t\(x\)-x$ is a diffeomorphism from $\R$ to itself, let $G$ be its inverse. Denote by $x^*$ the unique fixed point of $F_t$ and perform the change of variables "$u = x^* -F_t\(x\) +x$" to get
\begin{align*}
\textrm{tr}_\epsilon A_{\omega,\omega,t} = \s{\R}{h_\omega \circ G \(x^*-u\) \(\chi_{\overrightarrow{\omega},t} e^{g_{m,t}} \theta_\omega\) \circ \pi \circ F_t \circ G\(x^*-u\) \rho_\epsilon\(x^*-u\) \frac{\chi_\epsilon \circ G\(x^*-u\)}{1-F_t'\(G\(x^*-u\) \)}}{x} \\ \underset{ \epsilon \to 0}{\to} \textrm{tr}^{\flat} A_{\omega,\omega,t} =  \frac{h_\omega\(x^* \) \(\chi_{\overrightarrow{\omega},t} e^{g_{m,t}} \theta_\omega\) \circ \pi \(x^*\)}{1-F_t'\(x^*\)}.
\end{align*}
Recall that $F_t$ is the inverse of a lift of $T_t^m$ that extends $\kappa_\omega \circ \(\left. T_t^m \right|_{W_{\overrightarrow{\omega},t}}\) \circ \left. \pi \right|_{\kappa_\omega\(V_\omega\)}$, thus $\pi\(x^*\)$ is a fixed point of $T_t^m$. Since $T_t^m$ induces an expanding diffeomorphism from $W_{\overrightarrow{\omega},t}$ to $T_t^m\( W_{\overrightarrow{\omega},t}\)$, it has at most one fixed point in $W_{\overrightarrow{\omega},t}$. If $y$ is such a fixed point and $y \in V_\omega$ then $\kappa_\omega \(y\)$ is a fixed point of $F_t$, thus $y = \pi\(x^*\)$ and
\[
\textrm{tr}^{\flat} A_{\omega,\omega,t} = \frac{\chi_{\overrightarrow{\omega},t}\(y\) e^{g_{m,t}\(y\)} \theta_\omega\(y\)}{1-\(T_t^m\)'\(y\)^{-1}}
.\]
Otherwise, $h_\omega\(x^*\)= 0$ or $\chi_{\overrightarrow{\omega},t} \circ \pi \(x^*\) = 0$ and $ \textrm{tr}^{\flat} A_{\omega,\omega,t} = 0$. Finally, we always have
\[
\textrm{tr}^{\flat} A_{\omega,\omega,t} = \sum_{T_t^m\(x\) = x} \frac{\chi_{\overrightarrow{\omega},t}\(x\) e^{g_{m,t}\(x\)} \theta_\omega\(x\)}{1-\(T_t^m\)'\(x\)^{-1}}
.\]
Summing over $\omega \in \Omega$ and then $\overrightarrow{\omega} \in I$, we get
\[
\textrm{tr}^{\flat}\(\mathcal{K}_t^m\) = \sum_{T_t^m x = x} \frac{\exp\( g_{m,t}\(x\)\)}{1 - \(\(T_t^m\)'\(x\)\)^{-1}}
.\]
\end{proof}

Formula \eqref{trace} of Lemma \ref{ltrace} implies that the flat determinant of $\mathcal{K}_t$ is the dynamical determinant of Theorem \ref{main} (and is given by \eqref{determinant} in the application).

We want now to show that the product of "bounded" terms $\(\mathcal{K}_t^m\)_b$ of the decompositions \eqref{decomp} of large enough powers of $\mathcal{K}_t$ have almost no trace. That's the point of Lemma \ref{nullité}. To do that, we need first to state an abstract property of the flat trace. Notice that the convergence in weak operator topology is the convergence that appears in Lemma \ref{bor}.

\begin{lm}\label{sommation}
If $s \in \R$ and $\(u_k\)_{k \geqslant 0}$ is a sequence of bounded operators on $H^s$ such that the series $\sum_{k\geqslant 0} u_k$ converges in the weak operator topology. Then for all $\epsilon >0$ we have
\[
\textrm{tr}_\epsilon \( \sum_{k = 0}^{+ \infty} u_k \) = \sum_{k = 0}^{+ \infty} \textrm{tr}_\epsilon u_k
.\]
And the same is true replacing $H^s$ by $\B^s$.
\end{lm}

\begin{proof}
First, using Banach-Steinhauss Theorem twice, we find that there exists $M$ such that for all $n \in \N$, we have
\[
\left\| \sum_{k =0}^n u_k \right\|_{H^s \to H^s} \leqslant M
.\]

Then, write $J_\epsilon$ as a sum of rank one operators
\[
J_\epsilon = \sum_{m \geqslant 0} l_m \otimes x_m
\]
with $l_m \in \(H^{s}\)'$, $x_m \in H^s$ and
\[
\sum_{m \geqslant 0} \left\|l_m\right\|_{\(H^s\)'} \left\|x_m\right\|_{H^s} < + \infty
.\]
Thus for all $n \in \N$ we have
\begin{align*}
\sum_{k=0}^n \textrm{tr}_\epsilon u_k & = \sum_{k=0}^n \sum_{m \geqslant 0} l_m\(u_k\(x_m\)\) \\
 & = \sum_{m \geqslant 0} l_m\( \(\sum_{k=0}^n u_k \) \(x_m\)\).
\end{align*} 
For all $m \in \N$ and $n \in \N$ we have
\[
\left| l_m\( \(\sum_{k=0}^n u_k \) \(x_m\)\)\right| \leqslant M \left\|l_m\right\|_{\(H^s\)'}\left\|x_m\right\|_{H^s}
.\]
Thus by dominated convergence and convergence in the weak operator topology:
\[
\sum_{k=0}^{+ \infty} \textrm{tr}_\epsilon u_k = \sum_{m \geqslant 0} \lim_{n \to + \infty} l_m\( \(\sum_{k=0}^n u_k\) \(x_m\) \) = \sum_{m \geqslant 0} l_m\( \(\sum_{k=0}^{+ \infty} u_k\) \(x_m\)\) = \textrm{tr}_\epsilon \( \sum_{k=0}^{+ \infty} u_k\)
.\]
\end{proof}

\begin{lm}\label{nullité}
There is an integer $L$ such that if $m_1,\dots,m_J$ are integers greater than $L$ then for all $t \in K$
\[
\textrm{tr}^{\flat} \(\prod_{j=1}^J \(\mathcal{K}_t^{m_j}\)_b\) = 0
.\]
\end{lm}

\begin{proof}
Write $\prod_{j=1}^J \(\mathcal{K}_t^{m_j}\)_b$ as a matrix $\(B_{\omega,\omega',t}\)_{\omega,\omega' \in \Omega}$. From the construction of the $\(\mathcal{K}_t^{m_j}\)_b$, it comes that, for all $\omega \in \Omega$, the operator $B_{\omega,\omega,t} $ can be written as a sum (in weak operator topology) of terms of the form
\begin{equation}\label{terme}
\op\(\psi_{n_1}\) \M_1 \op\(\psi_{l_1}\) \dots \op\(\psi_{n_J}\) \M_J \op\(\psi_{l_J}\) 
\end{equation}
with the $\M_j$ as in the first part and
\begin{equation}\label{recu}
2^{n_j} \leqslant \lambda^{m_j} 2^{l_j+6} \leqslant \lambda^L 2^{l_j+6}
\end{equation}
for all $j \in \set{1,\dots,J} $, which implies
\begin{equation*}
n_j \leqslant L \log_2 \lambda + l_j + 6
.\end{equation*}
We shall show that, provided $\epsilon$ is small enough and $L$ large enough, the "epsilon trace" of all these operators is zero, which ends the proof with Lemma \ref{sommation}.

Let $u$ be the operator defined by \eqref{terme} composed by $J_\epsilon$. If there is $j \in \set{1,\dots, J-1}$ such that $\psi_{l_j} \psi_{n_j} = 0$ then $u$ is zero and so is its trace. Otherwise, for all $j \in \set{1,\dots,J-1}$ we have
\[
l_j \leqslant n_{j+1}+1
,\]
which leads with \eqref{recu} to
\begin{equation*}
n_1 \leqslant J\(L \log_2 \lambda + 1\) + l_J + 6
.\end{equation*}
Thus, provided $L$ is large enough
\begin{equation}\label{loin}
n_1 +1 \leqslant l_J -1
.\end{equation}
Suppose that $\phi$ is an eigenvector of $u$ corresponding to a non-zero eigenvalue. Then $\widehat{\phi}$ is supported in $\textrm{ supp } \psi_{n_1}$. But, provided $\epsilon$ is small enough, $\rho_{\epsilon}$ is supported in $\left[-1,1\right]$ and thus $\widehat{J_\epsilon \phi} = \widehat{\rho_{\epsilon}} \(\rho_{\epsilon} \ast \widehat{\phi}\)$ is supported in $\left[ -2^{n_1} - 1; 2^{n_1}+1\right]$ that doesn't intersect $\textrm{ supp } \psi_{l_J}$ according to \eqref{loin}. Consequently, $\phi$ must be zero, which is absurd, and thus the spectrum of $u$ is $\set{0}$. Since $u$ is nuclear of order $0$ ($J_\epsilon$ is), its trace is zero (see for example corollary 4 page 18 of the second part of \cite{Groth}).
\end{proof}

\section{Linear response \textit{a priori}}\label{5}

In this section, we prove some properties of the spectrum of the transfer operator $\mathcal{K}_t$, using methods introduced by Gouëzel, Keller and Liverani in \cite{Ke-Li} and \cite{GLK} (see also Paragraph A.3 in \cite{Bal2} for a sum-up). Notice that the spaces $\B^s$ are not the same as those used, for instance, in \cite{Bal2}. In particular, there is no compact injection of $\B^s$ in $\B^{s'}$ when $s'<s$, which implies for instance that we cannot use Hennion's theorem. However, the results of \cite{Ke-Li} and \cite{GLK} do not require compact injection and thus we can prove that the spectrum of the transfer operator has the expected behaviour on the Banach spaces $\B^s$.

First, we recall that the transfer operator has a spectral gap. We shall denote by $P_t$ the topological pressure associated with the dynamics $T_t$ (see \S 3.4 of \cite{Ruelle2} for a definition).

\begin{lm}\label{trouspec}
For all $t \in K$ and all $\frac{3}{2} < s < r-1$, the real number $e^{P_{t}\(g_{t}\)}$ is an eigenvalue of $\mathcal{K}_{t} : H^s \to H^s$. Moreover, this eigenvalue is simple and there is $\eta <  e^{P_{t} \(g_{t}\)}$ that does not depend of $s$ (but may depend of $t$) such that $e^{P_{t} \(g_{t}\)} $ is the only eigenvalue of modulus larger than $\eta$ of $\mathcal{K}_t$ acting on $\B^s$ .
\end{lm}

\begin{proof}
Just notice that $S$ maps $\B^s$ in $\mathcal{C}^{\frac{1}{2}}$, and so if $\phi$ is an eigenvector of $\mathcal{K}_t$ for a non-zero eigenvalue then $S \phi$ is an eigenvector of $\mathcal{L}_t$ acting on $\mathcal{C}^{\frac{1}{2}}$ for the same eigenvalue. Next, note that the lemma is true while replacing $\mathcal{K}_t$ by $\mathcal{L}_t$ and $\B^s$ by $\mathcal{C}^{\frac{1}{2}}$ or $\mathcal{C}^r$ (see for instance Theorem 3.6 of \cite{Ruelle3}).
Reciprocally, we know that $e^{P_t\(g_t\)}$ is an eigenvalue of $\mathcal{L}_t$ acting on $\mathcal{C}^r$. Let $\phi$ be a corresponding eigenvector. Then the coordinates of $P \phi$ are $\mathcal{C}^r$ and compactly supported, thus $P \phi \in \B^s$, and $P \phi$ is an eigenvector of $\mathcal{K}_t$ for the eigenvalue $e^{P_t\(g_t\)}$. 
\end{proof}

We need some technical estimates to apply the results from \cite{Ke-Li} and \cite{GLK}.

\begin{lm}\label{sousmult}
Let $ e^{P_0\(g_0\)} < R $. There is a constant $C>0$ such that for all $t \in K$ sufficiently close to $0$ and all $m \in \N^*$ we have
\begin{equation*}
Z\(t,m\) = \inf_{\beta \textup{ subcover of } \alpha_{m,t}} \sum_{V \in \beta} \exp \( \sup_{V} g_{m,t}\) \leqslant C R^m
.\end{equation*} 
\end{lm}

\begin{proof}
We know that $e^{P_0\(g_0\)} = \inf_{m \in \N^*} Z\(0,m\)^{\frac{1}{m}}$. Thus there exists $m_0 \in \N^*$ such that $Z\(0,m_0\) < R^{m_0}$. Since $Z\(.,m_0\)$ is upper semi-continuous, this inequality still holds replacing $0$ by $t$ sufficiently close. For such a $t$, we can write for $m = q m_0 + r $:
\[
Z\(t,m\) \leqslant Z\(t,m_0\)^q Z\(t,r\) \leqslant R^{q m_0} \sup_{k=0,\dots,q-1} \left\|Z\(.,k\)\right\|_{\infty,K} \leqslant C R^{m}
.\]
\end{proof}

\begin{lm}\label{toujours}
For each $\rho > \lambda e^{P_0\(g_0\)}$ there exists a neighbourhood $W$ of $0$ in $K$ such that for all $s,s' \in \left] \frac{3}{2}, r-1\right[$, there are some constants $C_1, C_2,M$ such that for all $m \in \N^*$ , $t \in W$ and $\phi \in \B^s$ we have
\begin{equation*}
\left\|\mathcal{K}_t^m \phi \right\|_{\B^s} \leqslant C_1 \rho^m \left\|\phi\right\|_{\B^s} + C_2 M^m \left\|\phi\right\|_{\B^{s'}}
.\end{equation*}
\end{lm}

\begin{proof}
Choose $\rho' \in \left] \lambda e^{P_0\(g_0\)}, \rho \right[$. Applying Lemma \ref{sousmult} and recalling Theorem \ref{fonda}, we have for $t$ sufficiently close to $0$ and $\sigma \in \set{s,s'}$
\begin{equation*}
\left\|\(\mathcal{K}_t^m\)_c\right\|_{\L\(\B^{\sigma} , \B^{\sigma}\)} \leqslant C_1 \(\rho'\)^m
\end{equation*}
for some constant $C$. Choose $L$ large enough so that $C_1 \(\rho'\)^L \leqslant \rho^L$ and then write for $m = qL+r \in \N^*$ (with $0 \leqslant r < L$)
\begin{equation*}
\mathcal{K}_t^m = \(\(\mathcal{K}_t^L\)_b\)^q \circ \(\mathcal{K}_t^r\)_b + \mathcal{K}_t^{qL} \circ \(\mathcal{K}_t^r\)_c + \sum_{k=0}^{q-1} \mathcal{K}_t^{kL} \circ  \(\mathcal{K}_t^L\)_c \circ \(\(\mathcal{K}_t^L\)_b\)^{q-k-1} \circ \(\mathcal{K}_t^r\)_b
.\end{equation*}
Thus if $\phi \in \B^s$ and $t$ sufficiently close to $0$, we have
\begin{align*}
\left\|\mathcal{K}_t^m \phi \right\|_{\B^s} & \leqslant C \rho^m \left\|\phi\right\|_{\B^s} \\ & + \( \left\|\mathcal{K}_t\right\|_{\L\(\B^s,\B^s\)}^{qL} \left\|\(\mathcal{K}_t^r\)_c \right\|_{\L\(\B^{s'}, \B^s\)} + \sum_{k=0}^{q-1} \left\|\mathcal{K}_t\right\|_{\L\(\B^s, \B^s\)}^{kL} \left\|\(\mathcal{K}_t^L\)_c\right\|_{\L\(\B^{s'}, \B^s\)}  c \rho_1^{\(q-k-1\)L +r}\) \left\|\phi\right\|_{\B^{s'}}
.\end{align*}
From which we get
\begin{equation*}
\left\|\mathcal{K}_t^m \phi \right\|_{\B^s}  \leqslant C_1 \rho^m \left\|\phi\right\|_{\B^s} + C_2 M^m \left\|\phi\right\|_{\B^{s'}}
\end{equation*}
for some constant $C_1, C_2, M$. 
\end{proof}

From now on, we suppose that $K$ is a rectangle. Fix $\eta> \lambda e^{P_0\(g_0\)} $ as in Lemma \ref{trouspec} for $t=0$ and $0 < \delta< e^{P_0\(g_0\)} - \eta $. Set
\begin{equation*}
V_{\delta,\eta} = \set{z \in \C : \left|z \right| \leqslant \eta \textrm{ or } \left| z - e^{P_0\(g_0\)} \right|  \leqslant \delta}
.\end{equation*}

We state a result of continuity and then a result of differentiability.

\begin{lm}\label{bound}\label{rescont}
If $N \geqslant 1$ then for all $ \frac{5}{2} < s < r-1 $ and all $\epsilon >0$, there is a neighbourhood $W$ of $0$ in $K$ such that for all $t \in W$, the spectrum of $\mathcal{K}_t$ acting on $\B^s$ is contained in $V_{\delta,\eta}$. Moreover there exists a constant $C$ such that for all $t \in W$ and $z \in \C \setminus V_{\delta,\eta}$ we have
\begin{equation}\label{bo}
\left\|\(z-\mathcal{K}_t\)^{-1}\right\|_{\L\(\B^s,\B^s\)} \leqslant C
\end{equation}
and the map
\begin{equation*}
t \mapsto \(z-\mathcal{K}_t\)^{-1} \in \L\(\B^s, \B^{s-\epsilon}\)
\end{equation*}
is continuous on $W$.
\end{lm}

\begin{proof}
We want to appply Theorem 1 of \cite{Ke-Li} twice with $\|.\| = \|.\|_{\B^s}$ and $|.| = \left\|.\right\|_{\B^{s-\epsilon}}$ and $\delta$ replaced by $\frac{\delta}{2}$. This theorem is stated in a one-dimensional setting but in view of the dependences in the data of the constants appearing in the results (that are explicitly given), it may be applied here. Thanks to Theorem \ref{fonda} and Lemma \ref{toujours}, there is a neighbourhood $W'$ of $0$ in $K$ such that the conditions (2) and (3) of \cite{Ke-Li} are fulfilled by the family $\(\mathcal{K}_t\)_{t \in W'}$(with $\alpha = \rho < \eta$). Since $t \mapsto \mathcal{K}_t \in \L\(\B^{s},\B^{s-1-\epsilon}\)$ is $\mathcal{C}^1$ on a neighbourhood of $K$, we get by interpolating between $\B^s$ and $\B^{s-1-\epsilon}$ (that is applying the inequality \eqref{interpolation})
\begin{equation*}
\left\|\mathcal{K}_{t_1} - \mathcal{K}_{t_0} \right\|_{L\(\B^s, \B^{s-\epsilon}\)} \leqslant C \left|t_1-t_0\right|^{\beta}
\end{equation*} 
for all $t_0,t_1 \in K$ and some constants $C$ and $\beta >0$. Thus the condition (5) of \cite{Ke-Li} is fulfilled on $W'$. As pointed out in Remark 6 of \cite{Ke-Li}, the condition (4) is unnecessary here since $\C \setminus V_{\frac{\delta}{2},\eta}$ is connected.

Thus applying this theorem and the remark, we find a neighbourhood $W$ of $0$ in $t$ such that for all $t \in W$ the spectrum of $\mathcal{K}_t$ acting on $\B^s$ is contained in $V_{\frac{\delta}{2},\eta}$. Now, we may apply the theorem from \cite{Ke-Li} again, taking each point of $W$ as the origin, to end the proof of the lemma.
\end{proof}

\begin{lm}\label{derres}
For all integers $0 \leqslant k \leqslant N-1$, all real $\frac{5}{2} + k < s < r-1$, and all $\epsilon >0$, there is a neighbourhood $W$ of $0$ in $K$ such that for all $z \in V_{\delta,\eta}$ the map
\begin{equation}\label{res}
t \mapsto \(z-\mathcal{K}_t\)^{-1} \in \L\(\B^s,\B^{s-k-\epsilon}\)
\end{equation}
is $\mathcal{C}^k$ on $W$. Moreover, for all multi-indices $\alpha$ with $\left|\alpha \right| \leqslant k$ we have
\begin{equation}\label{resder}
\frac{\partial^{\alpha}}{\partial t^{\alpha}} \(z-\mathcal{K}_t\)^{-1} = \sum_{\substack{\alpha_1,\dots,\alpha_j \neq 0 \\ \alpha_1 + \dots + \alpha_j=\alpha}} \(z-\mathcal{K}_t\)^{-1} \(\frac{\partial^{\alpha_1}}{\partial t^{\alpha_1}} \mathcal{K}_t\) \(z- \mathcal{K}_t\)^{-1} \dots \(\frac{\partial^{\alpha_j}}{\partial t^{\alpha_j}} \mathcal{K}_t\) \(z-\mathcal{K}_t\)^{-1}
.\end{equation}
\end{lm}

\begin{proof}
The case $k=0$ has been dealt with in Lemma \ref{rescont}. 

The case $k=1$ is a consequence of Theorem A.4 of \cite{Bal2} with the spaces $\B^0, \B^1,\B^2$ of \cite{Bal2} being here respectively $\B^s, \B^{s-1-\frac{\epsilon}{2}},\B^{s-2-\epsilon}$. There is a neighborhood $W$ of $0$ on which the hypotheses (A.2)-(A.7) are fulfilled (even when replacing $0$ by another element $t_0$ of $W$): (A.2) and (A.3) are consequences of Theorem \ref{fonda} and Lemma \ref{toujours}, (A.4) is contained in Lemma \ref{toujours}, (A.5) is implied by Taylor formula since $t \mapsto \mathcal{K}_t \in \L\(\B^s, \B^{s-1-\frac{\epsilon}{2}}\)$ is $\mathcal{C}^1$ (as pointed out in Remark \ref{transder}), (A.6) and (A.7) are a consequence of the fact that $t \mapsto \mathcal{K}_t \in \L\(\B^s, \B^{s-2-\epsilon}\)$ is $\mathcal{C}^2$. Applying this theorem in each direction and interpolating between $\mathcal{B}^{s-1-\frac{\epsilon}{2}}$ and $\mathcal{B}^{s-2-\epsilon}$, it comes that the map defined by \eqref{res} admits partial derivatives given by \eqref{resder} on a neighborhood of $0$. These partial derivatives are continuous on a neighborhood of $0$ as a consequence of Lemma \ref{rescont}. 

The end of the proof is an induction. We show show how to get $k=2$ from $k=1$. Notice that we cannot calculate the second partial derivatives by differentiating the partial derivatives as products. However, the expected formula will stand. Indeed, fix $i,j \in \set{1,\dots,d}$ and, setting $A_t = \(z-\mathcal{K}_t\)^{-1}$, write the growth rate
\begin{align*}
\frac{A_{t + h e_j}\drond{}{t_i} \mathcal{K}_{t+h e_j} A_{t +h ej} -  A_t \drond{}{t_i} \mathcal{K}_{t} A_t}{h}  - \( \drond{}{t_j} A_t \drond{}{t_i} \mathcal{K}_{t} A_t + A_t \ddeux{}{t_i}{t_j} \mathcal{K}_{t} A_t + A_t \drond{}{t_i} \mathcal{K}_{t} \drond{}{t_j}A_t\)
,\end{align*}
where $e_j$ is the j-th vector of the canonical basis of $\R^d$, as
\begin{align*}
& \(\frac{A_{t+he_j}-A_t}{h} - \drond{}{t_j} A_t\)\drond{}{t_i} \mathcal{K}_{t + h e_j} A_{t+he_j} 
+ A_t \( \frac{\drond{}{t_i} \mathcal{K}_{t+h e_j} - \drond{}{t_i} \mathcal{K}_t}{h} - \ddeux{}{t_i}{t_j} \mathcal{K}_t \) A_{t+h e_j} \\
& \quad + A_t \drond{}{t_i}\mathcal{K}_t \( \frac{A_{t+h e_j} - A_t}{h} - \drond{}{t_j}A_t \) \\
& \quad + \drond{}{t_j} A_t \( \drond{}{t_i} \mathcal{K}_{t + h e_j} A_{t+ h e_j} - \drond{}{t_i} \mathcal{K}_t A_t\) 
+ A_t \ddeux{}{t_i}{t_j} \mathcal{K}_t \(A_t - A_{t+h e_j}\).
\end{align*}
The norm of this expression as an operator from $\B^s$ to $\B^{s-2-\epsilon}$ is smaller than 
\begin{align*}
 & \left\| \frac{A_{t+he_j}-A_t}{h} - \drond{}{t_j} A_t\right\|_{\L\(\B^s, \B^{s_1}\)} \left\|\drond{}{t_i} \mathcal{K}_{t + h e_j} \right\|_{\L\(\B^{s_1}, \B^{s_2}\)} \left\|A_{t+h e_j}\right\|_{\L\(\B^{s_2}, \B^{s_2}\)}  \\
& \quad +  \left\|A_t\right\|_{\L\(\B^s, \B^s\)} \left\| \frac{\drond{}{t_i} \mathcal{K}_{t+h e_j} - \drond{}{t_i} \mathcal{K}_t}{h} - \ddeux{}{t_i}{t_j} \mathcal{K}_t \right\|_{\L\(\B^s, \B^{s_2}\)} \left\|A_{t+h e_j}\right\|_{\L\(\B^{s_2}, \B^{s_2}\)}  \\
& \quad +  \left\|A_t\right\|_{\L\(\B^s, \B^s\)} \left\|\drond{}{t_i}\mathcal{K}_t\right\|_{\L\(\B^s, \B^{s_1}\)} \left\|\frac{A_{t+h e_j} - A_t}{h} - \drond{}{t_j}A_t \right\|_{\L\(\B^{s_1}, \B^{s_2}\)}  \\
& \quad +  \left\|\drond{}{t_j} A_t\right\|_{\L\(\B^{s}, \B^{s_1}\)} \left\| \drond{}{t_i} \mathcal{K}_{t + h e_j} A_{t+ h e_j} - \drond{}{t_i} \mathcal{K}_t A_t\right\|_{\L\(\B^{s_1}, \B^{s_2}\)}  \\
& \quad +  \left\|A_t \ddeux{}{t_i}{t_j} \mathcal{K}_t\right\|_{\L\(\B^s, \B^{s_2'}\)} \left\|A_t - A_{t+h e_j} \right\|_{\L\(\B^{s_2'}, \B^{s_2}\)}, 
\end{align*}
where $s_1 = s-1-\frac{\epsilon}{2}$, $s_2 = s - 2 - \epsilon$ and $s_2' = s - 2 -\frac{\epsilon}{2}$. From the previous cases, we get that, provided $t$ is in some neighbourhood of $0$, this expression tends to $0$ as $h$ tends to $0$. Consequently, $A_t$ has the second partial derivatives announced. Since they are continuous, the result is proved for $k=2$. The idea of the proof in the general case is the same, up to more notational issues.
\end{proof}

We also need some information about the greatest eigenvalue of $\mathcal{K}_t$ and the associated spectral projection.

\begin{lm}\label{rg1}
For all $\frac{5}{2} < s < r-1$, there is a neighbourhood $W$ of $0$ in $K$ such that for all $t \in W$, $e^{P_t\(g_t\)} \in \D{e^{P_0\(g_0\)}}{ \delta}$ and $e^{P_t\(g_t\)}$ is the only element of the spectrum of $\mathcal{K}_t$ that has modulus greater than~$\eta$.
\end{lm}

\begin{proof}
Apply Corollary 1 of \cite{Ke-Li} together with Lemma \ref{trouspec}.
\end{proof}

\begin{lm}
For all $0 \leqslant k \leqslant N-1$, all $ \frac{5}{2}+k < s < r-1 $ and all $\frac{5}{2} < s' < r-1-k$, there is a neighbourhood $W$ of $0$ in $K$ such that for all $t \in W$ the spectrum of $\mathcal{K}_t$ is contained in $V_{\delta,\eta}$ and setting for all $t \in W$
\[
\Pi_t = \frac{1}{2i \pi} \s{\Gamma}{\(z-\mathcal{K}_t\)^{-1}}{z}
,\]
where $\Gamma$ is a circle of center $e^{P_0\(g_0\)}$ and of radius slightly larger than $\delta$, the map
\[
t \mapsto \Pi_t \in \L_{nuc}\(\B^s,\B^{s'}\)
\]
is $\mathcal{C}^k$ on $W$.
\end{lm}

\begin{proof}
The first point is only a reminder of Lemma \ref{bound}. Recall that $\Pi_t$ is a spectral projection. Choose $\epsilon >0$ sufficiently small and $\phi_0 \in \B^{s'+k+\epsilon} \cap \B^s $ such that $\Pi_0 \phi_0 \neq 0$ (for instance take an eigenvector of $\mathcal{K}_0$ for the eigenvalue $e^{P_0\(g_0\)}$). Then choose a linear form $l$ on $\mathcal{B}^0$ such that $l\(\Pi_0 \phi_0 \) \neq 0$. Set $\rho_t = \Pi_t \phi_0$ and notice that lemma \ref{derres} and dominate convergence imply that
\[
t \mapsto \rho_t \in \B^{s'}
\]
is $\mathcal{C}^k$ on a neighbourhood of $0$. In particular $\rho_t \neq 0$ for $t$ close enough of $0$. Applying Lemma \ref{rg1}, the spectral projector $\Pi_t$ has rank one when acting on $\B^s$, providing $t$ is small enough, and can thus be written 
\begin{equation}\label{rg12}
\Pi_t = m_t \otimes \rho_t
\end{equation} 
with $m_t \in \(\B^s \)'$. But for $t$ sufficiently close to $0$ so that $l\(\rho_t\) \neq 0$, $m_t$ may be written as
\[
m_t = \frac{l \circ \Pi_t}{l \(\rho_t\)}
\]
in which $\Pi_t$ is seen as an operator from $\B^{s}$ to $\B^0$, and thus has a $\mathcal{C}^k$ dependence in $t$. Consequently, $m_t$ has a $\mathcal{C}_k$ dependence in $t$ on a neighbourhood of $0$. Hence, the formula \eqref{rg12} ends the proof.
\end{proof}

\begin{lm}\label{rmoinsq}
If $r \geqslant 4$, then the map $ t \mapsto e^{P_t\(g_t\)} $ is $\mathcal{C}^{k}$ on $U$ where $k = \min\(N-1,r-4\)$.
\end{lm}

\begin{proof}
Choose $s \in \left]\frac{5}{2}+k, r-1\right[$, $\phi \in \B^s$ and $l \in \(\B^0\)'$ such that $l\(\Pi_0 \phi\) \neq 0$, write
\[
e^{P_t\(g_t\)} = \frac{l\(\mathcal{K}_t \Pi_t \phi\) }{l\(\Pi_t \phi\)} = \frac{l\( \s{\Gamma}{z\(z-\mathcal{K}_t\)^{-1}\phi }{z}\)}{l\(\s{\Gamma}{\(z-\mathcal{K}_t\)^{-1}\phi}{z}\)} 
\]
and apply Lemma \ref{derres} (of course, one can work the same way on the neighbourhood of any point of $U$).
\end{proof}

\section{Regularity of dynamical determinants}\label{6}

We are now ready to prove regularity for dynamical determinants.

\begin{thm}\label{main}
Let $U$ be an open set of $\R^D$ with $0 \in U$. Let $r \geqslant 4$ be an integer and $N \in \N^*$.  Let $t\mapsto T_t \in \mathcal{C}^{r+1}\(S^1,S^1\)$ be a $\mathcal{C}^N$ function on $U$ whose values are expanding maps with uniform expansion constant $\lambda^{-1} > 1$ (that is, \eqref{dilatation} holds). Let $t  \mapsto g_t \in \mathcal{C}^r\(S^1\)$ be a $\mathcal{C}^N$ function on $U$. Set $k = \min \( \left\lfloor \frac{r}{2} - \frac{7}{4}\right\rfloor, r-4, N-1\)$. Then there are $R > e^{-P_0\(g_0\)}$ and a neighbourhood $W$ of $0$ in $U$ such that:
\begin{enumerate}[label=\roman*.]
\item for all $t \in W$, the power series
\[
\sum_{n \geqslant 1} \frac{1}{n} \sum_{T_t^n x = x} \frac{\exp\(g_{n,t}\(x\)\)}{1-\( \(T_t^n\)'\(x\) \)^{-1}} z^n
,\]
where $g_{n,t}$ is defined by \eqref{gmt}, has a non-zero convergence radius;
\item for all $t \in W$ the map
\[
z \mapsto \exp\(-\sum_{n \geqslant 1} \frac{1}{n} \sum_{T_t^n x = x} \frac{\exp\(g_{n,t}\(x\)\)}{1-\( \(T_t^n\)'\(x\) \)^{-1}} z^n\)
\]
extends on $\D{0}{R}$ to a holomorphic function $z \mapsto d\(z,t\)$;
\item for all $t \in W$, the real number $e^{-P_t\(g_t\)}$ is the unique zero of $d\(.,t\)$ and it is simple;
\item the map $ \(z,t\)  \mapsto d\(z,t\) $ is $\mathcal{C}^k$ on $\D{0}{R} \times W$.
\end{enumerate}
\end{thm}

\begin{proof}
For the sake of simplicity, we write the proof when $k=1$. Thus we have $r \geqslant 6$. 

Set $s_0 = 2,6$, $s_1= 3,6$, $s_2=3,75$, $s_3=4$ and $s_4 = 4,8$, and choose a neighbourhood $W$ of $0$ in $K$ such that there is a constant $C$ such that
\begin{enumerate}[label=\arabic*.]
\item for all $t \in W$, the spectrum of $\mathcal{K}_t $ acting on $\B^{s_0}$, on $\B^{s_2}$, on $\B^{s_3}$ and on $\B^{s_4}$ is contained in $V_{\delta,\eta}$;
\item for all $t \in W$ and $z \in \C \setminus V_{\delta,\eta}$:
\begin{equation*}
\left\| \(z-\mathcal{K}_t\)^{-1} \right\|_{\B^{s_2} \to \B^{s_2}} \leqslant C
;\end{equation*}
\item for all $z \in \C \setminus V_{\delta,\eta} $ the map
\[
t \mapsto \(z-\mathcal{K}_t\)^{-1} \in \L\(\B^{s_3},\B^{s_2}\)
\] 
is continuous on $W$ with $\mathcal{C}^0$ norm bounded by $C$;
\item for all $z \in \C \setminus V_{\delta,\eta} $ the map
\[
t \mapsto \(z-\mathcal{K}_t\)^{-1} \in \L\(\B^{s_4},\B^{s_2}\)
\] 
is $\mathcal{C}^1$ on $W$ with $\mathcal{C}^1$ norm bounded by $C$;
\item for all $t \in W$ the spectral projection $\Pi_t :\B^{s_0} \to \B^{s_1}$ has rank one;
\item the map
\[
t \mapsto \Pi_t \in \L_{nuc}\(\B^{s_0}, \B^{s_4}\)
\] 
is continuous on $W$ with $\mathcal{C}^0$ norm bounded by $C$;
\item the map 
\[
t \mapsto \Pi_t \in \L_{nuc}\(\B^{s_1},\B^{s_3}\)
\]
is $\mathcal{C}^1$ on $W$ with $\mathcal{C}^1$ norm bounded by $C$;
\item there is $\rho_0 \in \left] \lambda e^{P_0\(g_0\)}, e^{P_0\(g_0\)} \right[$ such that for all $t \in W$, $m \in \N^*$ and $i \in \set{0;2}$
\begin{equation*}
\left\|\(\mathcal{K}_t^m\)_b\right\|_{\L\(\B^{s_i}, \B^{s_i}\)} \leqslant C \rho_0^m.
\end{equation*}
\end{enumerate}
This may be done thanks to the results of \S \ref{5}.

Than for all $t \in W$ write $\mathcal{K}_{0,t} = \Pi_t \mathcal{K}_t = e^{P_t\(g_t\)} \Pi_t$ and $\mathcal{K}_{1,t} = \(Id - \Pi_t\) \mathcal{K}_t$, where $\mathcal{K}_t$ is seen as acting on $\B^{s_2}$, then we have
\[
\mathcal{K}_t = \mathcal{K}_{0,t} + \mathcal{K}_{1,t} \textrm{ and } \mathcal{K}_{1,t} \mathcal{K}_{0,t} = \mathcal{K}_{0,t} \mathcal{K}_{1,t} = 0
\]
and for all $m \in \N^*$ the flat traces of $\mathcal{K}_t^m$ and of $\mathcal{K}_{0,t}^m$ (which has rank one) are well-defined and thus the trace of $\mathcal{K}_{1,t}^m$ too. Consequently, we have
\begin{equation*}
d\(z,t\) = \textrm{det}^{\flat} \(I-z \mathcal{K}_t\) = \textrm{det}^{\flat}\(I-z \mathcal{K}_{0,t}\) \textrm{det}^{\flat}\(I- z \mathcal{K}_{1,t}\) = \(1-z e^{P_t\(g_t\)}\) \textrm{det}^{\flat}\(I- z \mathcal{K}_{1,t}\)
\end{equation*}
as formal power series. Then set, for all $m \in \N^*$ and $t \in W$, $h_m\(t\) = \textrm{tr}^{\flat}\(\mathcal{K}_{1,t}^m\) $. We want to show that the $\mathcal{C}^1$ norm of $h_m$ is bounded by $C' \rho^m$ for some constants $C' >0$ and $\rho < e^{P_0\(g_0\)}$. To do that, write for all $m \in \N^*$ and all $t \in W$
\begin{equation*}
\mathcal{K}_{1,t}^m = \(\mathcal{K}_t^m\)_b + \(\mathcal{K}_{1,t}^m\)_{cc} \textrm{ where } \(\mathcal{K}_{1,t}^m\)_{cc} = \mathcal{K}_{1,t}^m - \(\mathcal{K}_t^m\)_b = \(\mathcal{K}_t^m\)_c - \mathcal{K}_{0,t}^m
.\end{equation*}
Then choose $\rho_1 \in \left]\rho_0, e^{P_0\(g_0\)}\right[$ and an integer $L$ large enough so that $c \rho_0^L \leqslant \rho_1^L$ and Lemma \ref{nullité} holds for $s = s_2$. Then if $m \geqslant L$ write $m = qL + r$ with $L \leqslant r < 2L$ and then
\begin{equation*}
\mathcal{K}_{1,t}^m = \(\mathcal{K}_t^L\)_b^q \(\mathcal{K}_t^r\)_b + \mathcal{K}_{1,t}^{qL} \(\mathcal{K}_{1,t}^r\)_{cc} + \sum_{k=0}^{q-1} \mathcal{K}_{1,t}^{kL} \(\mathcal{K}_t^L\)_{cc} \(\mathcal{K}_t^L\)_b^{q-k-1} \(\mathcal{K}_t^r\)_b
.\end{equation*}
This implies with Lemma \ref{nullité}
\begin{equation*}
h_m\(t\) = \textrm{tr}\( \mathcal{K}_{1,t}^{qL} \(\mathcal{K}_{1,t}^r\)_{cc} \) + \sum_{k=0}^{q-1} \textrm{tr}\( \mathcal{K}_{1,t}^{kL} \(\mathcal{K}_t^L\)_{cc} \(\mathcal{K}_t^L\)_b^{q-k-1} \(\mathcal{K}_t^r\)_b \)
.\end{equation*}
As in the proof of Lemma \ref{derres}, one may prove using in particular properties 3,6 and 7 above that the nuclear operator from $\B^{s_2}$ to itself $\mathcal{K}_{1,t}^{qL} \(\mathcal{K}_{1,t}^r\)_{cc} $ has a $\mathcal{C}^1$ dependence in $t$ with for $i \in \set{1,\dots,d}$:
\begin{align*}
\left\| \drond{}{t_i}\(\mathcal{K}_{1,t}^{qL} \(\mathcal{K}_{1,t}^r\)_{cc} \) \right\|_{\L_{nuc}\(\B^{s_2} , \B^{s_2}\)} & = \left\| \drond{}{t_i}  \(\mathcal{K}_{1,t}^{qL}\) \(\mathcal{K}_{1,t}^r\)_{cc}  +  \mathcal{K}_{1,t}^{qL} \drond{}{t_i}\(\(\mathcal{K}_{1,t}^r\)_{cc} \) \right\|_{\L_{nuc}\( \B^{s_2} , \B^{s_2}\)} \\
  & \leqslant \frac{1}{2 \pi} \left\| \s{\gamma}{z^{qL} \drond{}{t_i} \( \(z- \mathcal{K}_t\)^{-1}\)}{z} \right\|_{\L\(\B^{s_4}, \B^{s_2}\)} \left\| \(\mathcal{K}_{1,t}^r\)_{cc} \right\|_{\L_{nuc}\( \B^{s_2} , \B^{s_4}\)} \\ & +  \frac{1}{2 \pi} \left\| \s{\gamma}{z^{qL} \(z- \mathcal{K}_t\)^{-1}}{z}\right\|_{\L\(\B^{s_2},\B^{s_2}\)} \left\| \drond{}{t_i}\( \(\mathcal{K}_{1,t}^r\)_{cc}  \)\right\|_{\L_{nuc}\( \B^{s_2}, \B^{s_2}\)} \\
  & \leqslant C' \rho_2^m,
\end{align*}
where $\gamma$ is a circle of center $0$ and of radius $\rho_2$ slightly greater than $\eta$ and $C'$ is some constant. If $k \in \set{0,\dots, q-1}$ we can write in the same way $\drond{}{t_i} \( \mathcal{K}_{1,t}^{kL} \(\mathcal{K}_t^L\)_{cc} \(\mathcal{K}_t^L\)_b^{q-k-1} \(\mathcal{K}_t^r\)_b\) $ in the same way as
\begin{align*}
 &  \frac{1}{2i \pi} \(\s{\gamma}{z^{kL} \drond{}{t_i} \( \(z-\mathcal{K}_t\)^{-1} \)}{z}\) \(\mathcal{K}_t^L\)_{cc} \(\mathcal{K}_t^L\)_b^{q-k-1} \(\mathcal{K}_t^r\)_b \\
  & \quad + \frac{1}{2i \pi} \(\s{\gamma}{ z^{kL} \(z-\mathcal{K}_t\)^{-1}}{z} \) \drond{}{t_i}\(\(\mathcal{K}_t^L\)_{cc} \) \(\mathcal{K}_t^L\)_b^{q-k-1} \(\mathcal{K}_t^r\)_b \\
  & \quad + \sum_{j=0}^{q-k-1} \frac{1}{2i \pi} \(\s{\gamma}{ z^{kL} \(z-\mathcal{K}_t\)^{-1}}{z}\) \(\mathcal{K}_t^L\)_{cc} \(\mathcal{K}_t^L\)_b^j \drond{}{t_i} \( \(\mathcal{K}_t^L \)_b\) \(\mathcal{K}_t^L\)_b^{q-k-j-2} \(\mathcal{K}_t^r\)_b \\
  & \quad + \mathcal{K}_{1,t}^{kL} \(\mathcal{K}_t^L\)_{cc} \(\mathcal{K}_t^L\)_b^{q-k-1} \drond{}{t_i} \(\(\mathcal{K}_t^r\)_b\)
\end{align*}
and thus
\begin{align*}
& \left\| \drond{}{t_i} \( \mathcal{K}_{1,t}^{kL} \(\mathcal{K}_t^L\)_{cc} \(\mathcal{K}_t^L\)_b^{q-k-1} \(\mathcal{K}_t^r\)_b\) \right\|_{\L_{nuc}\( \B^{s_2}, \B^{s_2}\)}  \\ 
& \leqslant \frac{1}{2 \pi} \left\| \s{\gamma}{z^{kL} \drond{}{t_i} \( \(z-\mathcal{K}_t\)^{-1} \)}{z} \right\|_{\L\(\B^{s_4}, \B^{s_2}\)} \left\| \(\mathcal{K}_t^L\)_{cc} \right\|_{\L_{nuc}\( \B^{s_2} , \B^{s_4}\)}  \left\| \(\mathcal{K}_t^L\)_b \right\|_{\L\(\B^{s_2}, \B^{s_2}\)}^{q-k-1} \left\| \(\mathcal{K}_t^r\)_b \right\|_{\L\(\B^{s_2}, \B^{s_2}\)} \\
&  + \frac{1}{2 \pi} \left\| \s{\gamma}{z^{kL} \(z-\mathcal{K}_t\)^{-1}}{z} \right\|_{\L\(\B^{s_2}, \B^{s_2}\)} \left\| \drond{}{t_i}\(\(\mathcal{K}_t^L\)_{cc} \)  \right\|_{\L_{nuc}\( \B^{s_2}, \B^{s_2}\)}  \left\| \(\mathcal{K}_t^L\)_b \right\|_{\L\(\B^{s_2},\B^{s_2}\)}^{q-k-1} \left\| \(\mathcal{K}_t^r\)_b \right\|_{\L\(\B^{s_2}, \B^{s_2}\)} \\
 & + \sum_{j=0}^{q-k-1} \( \frac{1}{2 \pi}  \left\| \s{\gamma}{z^{kL} \(z-\mathcal{K}_t\)^{-1}}{z} \right\|_{\L\(\B^{s_2}, \B^{s_2}\)} \left\| \(\mathcal{K}_t^L\)_{cc} \right\|_{\L_{nuc}\(\B^{s_0}, \B^{s_2}\)} \left\| \(\mathcal{K}_t^L\)_b \right\|_{\L\(\B^{s_0}, \B^{s_0}\)}^j \right. \\ & \left. \times \left\|\drond{}{t_i} \(\(\mathcal{K}_t^L\)_b\) \right\|_{\L\(\B^{s_2}, \B^{s_0}\)} \left\| \(\mathcal{K}_t^L\)_b\right\|_{\L\(\B^{s_2}, \B^{s_2}\)}^{q-k-j-2} \left\|\(\mathcal{K}_t^r\)_b \right\|_{\L\(\B^{s_2}, \B^{s_2}\)} \) \\
 & + \frac{1}{2 \pi} \left\| \s{\gamma}{z^{kL} \(z-\mathcal{K}_t\)^{-1}}{z} \right\|_{\L\(\B^{s_2},\B^{s_2}\)} \left\| \(\mathcal{K}_t^L\)_{cc} \right\|_{\L_{nuc}\( \B^{s_0} , \B^{s_2}\)}  \left\|\(\mathcal{K}_t^L\)_b \right\|_{\L\(\B^{s_0},\B^{s_0}\)} \left\| \drond{}{t_i} \(\(\mathcal{K}_t^r\)_b \)\right\|_{\L\(\B^{s_2}, \B^{s_0}\)} \\ 
 & \leqslant C_2 m \rho_3^m
\end{align*}
where $\rho_3 = \max \(\rho_2,\rho_0\)$ and $C^2$ is some constant. Consequently we have for $t \in W$ and $m \geqslant L$
\begin{equation*}
\left| \drond{}{t_i} h_m\(t\) \right| \leqslant C_2 m^2 \rho_3^m \leqslant C \rho^m
.\end{equation*}
This ends the proof of the case $k=1$.

To deal with the general case let $\mathcal{K}_t$ act on $\B^s$ with $s$ a real number which satisfies $ \frac{5}{2}+k < s < r-1-k$ and write out the higher order partial derivatives of $h_m$ using the same tricks.
\end{proof}

Finally, we briefly recall the proof of Lemma 3.1 of \cite{Poll} to get formulae \eqref{valeur} and \eqref{derivee} in a differentiable setting.

\begin{cor}\label{adaptation}
Let $\tau \mapsto T_\tau$ be a $\mathcal{C}^2$ curve, defined on a neighbourhood $\left]-\epsilon,\epsilon \right[$ of $0$, of $\mathcal{C}^7$ expanding maps of the circle. Let $g : S^1 \to \R$ be a $\mathcal{C}^6$ map. Then there is some $R >1$ such that for all $u,t$ sufficiently close to $0$ the map defined by \eqref{determinant} extends on the disc of center $0$ and radius $R$ to a holomorphic function $z \mapsto d\(z,u,t\)$. The function $d$ defined this way is $\mathcal{C}^1$ and the formula \eqref{valeur} holds for sufficiently small $t$.

If $\tau \mapsto T_\tau$ is a $\mathcal{C}^3$ curve of $\mathcal{C}^8$ expanding maps of the circle and $g$ is $\mathcal{C}^7$, then $d$ is $\mathcal{C}^2$ and the formula \eqref{derivee} for linear response holds.
\end{cor}

\begin{proof}
The definition and regularity of the function $d$ are immediate consequences of Theorem \ref{main} taking $U= \left]-\epsilon,\epsilon\right[ \times \R$, $t=\(\tau,u\)$, $g_{\tau,u} = - u g - \log \left|T_\tau'\right|$ and $T_{\tau,u} = T_\tau$.

The end of the proof is just a reminder of the proof of Lemma 3.1 in \cite{Poll}. 

Write $z\(u,\tau\) = \exp\(- P_\tau\(-u g - \log \left|T_\tau'\right|\)\)$ (with $P_\tau = P_t$ the topological pressure of the dynamics $T_\tau$). For $u,t$ sufficiently small, $z\(u,\tau\)$ is in the disc of radius $R$ and center $0$. Notice that the regularity of the function $u,t \mapsto z\(u,t\)$ can be seen as a consequence of Lemma \ref{rmoinsq} or of the implicit function theorem. Differentiating the equation $d\(z\(u,\tau\),u,\tau\) = 0$ we get $\eqref{valeur}$ recalling that
\[
\left. \drond{}{u}\(P_\tau\(-u g - \log \left|T_\tau'\right|\)\)\right|_{u=0} = - \s{S^1}{g}{\mu_\tau}
.\]
See for example \S 7.28 ans \S 7.12 in \cite{Ruelle2} for a proof of this formula. The formula \eqref{derivee} for linear response is obtained by differentiating \eqref{valeur}.
\end{proof}

\begin{rmq}\label{fin}
Suppose that the hypotheses of Corollary \ref{adaptation} hold. Let $d$ be the function introduced in Corollary \ref{adaptation}. Write it as the sum of a power series
\begin{equation*}
d\(z,u,\tau\) = \sum_{n \geqslant 0} a_n\(u,\tau\) z^n
.\end{equation*}
Notice that this sum converges as soon as $\left|z\right| < R$. Then the formula \eqref{derivee} may be written
\begin{equation}\label{explicite}
\left. \drond{}{\tau}\( \s{S^1}{g}{\mu_\tau} \) \right|_{\tau=0} = - \frac{\sum_{n \geqslant 1} \ddeux{a_n}{u}{\tau}\(0,0\)}{\sum_{n \geqslant 1}n a_n\(0,0\)} + \frac{\(\sum_{n \geqslant 1} n \drond{a_n}{\tau}\(0,0\)\) \( \sum_{n \geqslant 1} \drond{a_n}{u}\(0,0\)\)}{\( \sum_{n \geqslant 1} n a_n\(0,0\) \)^2}.
\end{equation}

As pointed out in the paragraph 4.1 of \cite{Poll}, one can explicitly calculate the terms of the series that appear in \eqref{explicite} in terms of the periodic points of $T_0$ and of the derivatives of $\tau \mapsto T_\tau$ at $0$. To make the calculation easier we shall suppose that the $T_\tau$ are orientation-preserving (we get the orientation-reversing case by some sign changes). Then, setting 
\[
b_n\(u,\tau\) = \sum_{T_\tau^n x = x} \frac{\exp\(-u \sum_{k=0}^{n-1} g\(T_\tau^k x\)\)}{\left| \(T_\tau^n\)'\(x\)-1\right|} = \sum_{T_\tau^n x = x} \frac{\exp\(-u \sum_{k=0}^{n-1} g\(T_\tau^k x\)\)}{\(T_\tau^n\)'\(x\)-1} 
,\]
we find, using formulae \eqref{flatdet}, \eqref{coeffdet} and \eqref{trace}, that the $a_n$ are recursively defined by
\[
a_0 = 1 \textrm{ and } a_n  = - \frac{1}{n} \sum_{j=0}^{n-1} a_j b_{n-j} 
.\]
We immediately deduce similar formulae for the derivatives of the $a_n$:
\[
\drond{a_n}{u} = \drond{a_n}{\tau} = \ddeux{a_n}{u}{\tau} =0,
\]
\[
\drond{a_n}{u}  = - \frac{1}{n} \sum_{j=0}^{n-1} \(\drond{a_j}{u} b_{n-j}+ a_j \drond{b_{n-j}}{u} \), \quad
\drond{a_n}{\tau}  = - \frac{1}{n} \sum_{j=0}^{n-1} \(\drond{a_j}{\tau} b_{n-j}+ a_j \drond{b_{n-j}}{\tau} \),
\]
and
\[
\ddeux{a_n}{u}{\tau}  = - \frac{1}{n} \sum_{j=0}^{n-1} \( \ddeux{a_j}{u}{\tau} b_{n-j}  + \drond{a_j}{u}\drond{b_{n-j}}{\tau} + \drond{a_j}{\tau} \drond{b_{n-j}}{u} + a_j \ddeux{b_{n-j}}{u}{\tau} \).
\]
In order to apply these formulae, we need to calculate the derivatives of $b_n$. To do so we set: $X_n = \left. \drond{}{\tau} \(T_\tau^n\) \right|_{\tau = 0}$ for all non-zero integer $n$ and $X = X_1$. Then we have
\[
X_n = \sum_{k=0}^{n-1} X \circ T_0^k \(T_0^{n-1-k}\)' \circ T_0^k,
\]
\[
\drond{b_n}{u}\(0,0\) = - \sum_{T_0^n x = x}  \frac{\sum_{k=0}^{n-1} g\(T_0^k x\)}{\(T_0^n\)'\(x\)-1}, \quad
\drond{b_n}{\tau}\(0,0\) = - \sum_{T_0^n x = x} \frac{X_n'\(x\) + \(T_0^n\)'' \(x\) \frac{X_n\(x\)}{1 - \(T_0^n\)' \(x\) }}{\(\(T_0^n\)' \(x\) -1\)^2}
\]
and
\[
\ddeux{b_n}{u}{\tau}\(0,0\) = \sum_{T_0^n x = x} \( \frac{\sum_{k=0}^{n-1} g \circ T_0^k }{\(\(T_0^n\)'-1\)^2} \(\(T_0^n\)'' \frac{X_n}{1-\(T_0^n\)'} + X_n'\) + \frac{1}{\(T_0^n\)'-1} \sum_{k=0}^{n-1} \frac{X_n \circ T_0^k g' \circ T_0^k}{\(T_0^n\)' \circ T_0^k - 1} \)\(x\)
.\]

Finally using Cauchy's formula and differentiation under the integral, the partial derivatives of $d$ are holomorphic on the disc of center and radius $R$. Thus the power series that appear in formula \eqref{explicite} converge exponentially fast (we get a much faster convergence in the analytic setting, see \cite{Poll}).
\end{rmq}

\bibliographystyle{plain}
\bibliography{biblio}

\end{document}